\documentclass[10pt]{amsart}
\usepackage{amssymb,amsbsy,amsmath,amsfonts,times}
\usepackage{latexsym,euscript,exscale}

\usepackage{helvet}

\title{Displaying Polish groups on separable Banach spaces}
\author{Valentin Ferenczi  and Christian Rosendal }
\address{Instituto de Matem\'atica e Estat\'istica \\
 Universidade de S\~ao Paulo \\
rua do Mat\~ao 1010 \\
Cidade Universit\'aria \\
05508-90 S\~ao Paulo, SP \\
Brazil.}
\email{ferenczi@ime.usp.br}
\address{Department of Mathematics, Statistics, and Computer Science (M/C 249)\\
University of Illinois at Chicago\\
851 S. Morgan St.\\
Chicago, IL 60607-7045\\
USA}
\email{rosendal@math.uic.edu}
\urladdr{http://www.math.uic.edu/$~$rosendal}

\date{}
\newcommand{\norm}[1]{\lVert#1\rVert}
\newcommand{\Norm}[1]{\big\lVert#1\big\rVert}
\newcommand{\NORM}[1]{\Big\lVert#1\Big\rVert}
\newcommand{\triple}[1]{|\!|\!|#1|\!|\!|}
\newcommand{\Triple}[1]{\big|\!\big|\!\big|#1\big|\!\big|\!\big|}

\newcommand {\1}{\mathbb 1}

\newcommand {\N}{\mathbb N}
\newcommand {\Q}{\mathbb Q}
\newcommand {\R}{\mathbb R}
\newcommand {\Z}{\mathbb Z}
\newcommand {\C}{\mathbb C}
\newcommand {\U}{\mathbb U}
\newcommand {\T}{\mathbb T}

\newcommand{\eps}{\epsilon}

\newcommand{\isom}{\simeq}

\newcommand{\tom} {\emptyset}

\newcommand{\saa}{\Rightarrow}
\newcommand{\equi}{\Longleftrightarrow}

\newcommand{\til}{\rightarrow}

\newcommand {\del}{ \; \big| \;}

\newcommand{\inv}{^{-1}}

\newcommand {\e} {\exists}
\renewcommand {\a} {\forall}

\newtheorem{thm}{Theorem}
\newtheorem{cor}[thm]{Corollary}
\newtheorem{lemme}[thm]{Lemma}
\newtheorem{prop} [thm] {Proposition}
\newtheorem{defi} [thm] {Definition}
\newtheorem{obs}[thm] {Observation}

\newtheorem{quest}[thm]{Question}
\newtheorem{ex}[thm]{Example}

\newcommand{\Id}{{\rm Id}}

\usepackage{fouriernc}

\thanks{V. Ferenczi acknowledges the support of CNPq, project 452068/2010-0, and FAPESP, projects 2010/05182-1 and 2010/17493-1. C. Rosendal was supported by NSF grant DMS 0901405}

\begin{document}


\begin{abstract} A display of a topological group $G$ on a Banach space $X$ is a topological isomorphism of $G$ with the isometry group ${\rm Isom}(X,\triple\cdot)$ for some equivalent norm $\triple\cdot$ on $X$, where the latter group is equipped with the strong operator topology.

Displays of Polish groups on separable real spaces are studied. It is proved that any closed subgroup of the infinite symmetric group $S_\infty$ containing a non-trivial central involution admits a display on any of the classical spaces $c_0$, $C([0,1])$, $\ell_p$ and $L_p$ for $1 \leqslant p <\infty$. Also, for any Polish group $G$, there exists a separable space $X$ on which  $\{-1,1\} \times G$ has a display. 

\end{abstract}
\maketitle

\tableofcontents

\section{Introduction}

The general objective of this paper is to determine, given a topological group $G$ and a Banach space $X$, whether $G$ is isomorphic to the group ${\rm Isom}(X,{\triple{\cdot}})$ of isometries on $X$ equipped with the strong operator topology, under some adequate choice of equivalent norm ${\triple{\cdot}}$ on $X$. Recall that the {\em strong operator topology} (or \textsc{\textsc{sot}}), i.e., the topology of pointwise convergence on $X$, is a group topology on ${\rm Isom}(X)$, meaning that the group operations are continuous. Moreover, when $X$ is separable, ${\rm Isom}(X)$ is a {\em Polish} group with respect to this topology, that is, it is separable and the topology can be induced by a complete metric on ${\rm Isom}(X)$.
General results about isometry groups as Polish groups may be found in \cite{kec}.

The first important general results in the direction of this paper are due to K. Jarosz \cite{J}, who, improving earlier results by S. Bellenot \cite{Bel}, proved that every real or complex Banach space may be renormed so that the only isometries are scalar multiples of the identity. Therefore, no information can be deduced from the fact that $X$ has only trivial isometries about the isomorphic structure of a Banach space $X$, i.e., the structure invariant under equivalent renormings, such as, for example, super-reflexivity as opposed to uniform convexity. 

Since our subject matter in many ways touches the classical topic of linear representations, we need a few definitions to dispel any possible confusions in terminology.

By a  {\em (linear) representation} of a topological group $G$ on a Banach space $X$ we understand a homomorphism $\rho\colon G\til GL(X)$ such that for any $x\in X$, the map $g\in G\mapsto \rho(g)(x)\in X$ is continuous, or, equivalently, such that $\rho$ is continuous with respect to the strong operator topology on $GL(X)$. The representation is {\em faithful} if it is injective and {\em topologically faithful} if $\rho$ is a topological isomorphism between $G$ and its image $\rho(G)$.

The even stronger concept that we shall be studying here is given by the following definition.

\begin{defi}  
A {\em display} of a topological group $G$ on a Banach space $X$ is a topologically faithful linear representation $\rho \colon G\til GL(X)$ such that $\rho(G)={\rm Isom}(X,\triple\cdot)$ for some equivalent norm $\triple\cdot $ on $X$. 
We say that  $G$ is {\em displayable} on $X$ whenever there exists a display of $G$ on $X$. 
\end{defi}

Using this terminology, the result of Jarosz simply states that the group $\{-1,1\}$ is displayable on any real Banach space, while the unit circle $\T$ is displayable on any complex space.
It is then quite natural to ask  when a more general group $G$ is displayable on a  Banach space $X$, and, more specifically, if some information on the isomorphic structure of a space $X$ may be recovered from which groups $G$ are displayable on it.
In \cite{st} , J. Stern observed that a necessary condition for $G$ being displayable on a real Banach space is that $G$ contains a normal subgroup with $2$ elements, corresponding to the subgroup $\{\Id,-\Id\}$ of ${\rm Isom(X)}$ - or equivalently, that the center of $G$ contains a non-trivial involution, corresponding to $-\Id$. On the other hand, K. Jarosz \cite{J} conjectured  that for any discrete group $G$ and any real Banach space $X$ such that $\dim X \geqslant |G|$, the group $\{-1,1\} \times G$ is abstractly isomorphic to ${\rm Isom}(X,\triple\cdot)$ for some equivalent renorming $\triple\cdot$ of $X$.
A counter-example to this is the group ${\rm Homeo}_+(\R)$ of all increasing homeomorphisms of the unit interval equipped with the discrete topology. While it has cardinality equal to that of the continuum, it is shown in \cite{rossol} (see Corollary 9) that it admits no non-trivial linear representation on a separable reflexive Banach space and, in particular, cannot be the isometry group of such a space. Another and perhaps more striking example is the construction in \cite{FR} of a separable, infinite-dimensional, reflexive, real Banach space $X$ such that for any equivalent norm on $X$, one can decompose $X$ into an isometry invariant direct sum $X=H\oplus F$, where any isometry of $X$ acts by scalar multiplication on $H$. It follows that the isometry group is a compact Lie group under any equivalent renorming and, in particular, is never countably infinite. So not even the condition ${\rm dens}(X) \geqslant |G|$ can guarantee that
$\{-1,1\} \times G$ is isomorphic to ${\rm Isom}(X,\triple\cdot)$ for some equivalent norm $\triple\cdot$.

In a more positive direction, partial answers to Jarosz' conjecture were obtained by the first author and E. Galego in \cite{FG}. There it was shown that for any finite group $G$, the group $\{-1,1\} \times G$ is displayable on any  separable real Banach space $X$ for which $\dim X \geqslant |G|$ and it was also noted that it is not possible to generalize this result to all finite groups with a non-trivial central involution. On the other hand,  any countable discrete group $G$ containing a non-trivial central involution  is displayable on any of the spaces $c_0$,  $C([0,1])$, $\ell_p$ and $L_p$, $1 \leqslant p<\infty$.  

In this paper, we shall generalize results of \cite{FG} to certain uncountable Polish groups $G$, provided $G$ admits a topologically faithful linear representation on a space $X$ inducing a sufficient number of countable orbits on $X$. Among our applications, we deduce that any closed subgroup of the group $S_\infty$ of all permutations of $\N$ is displayable on any of the classical spaces $c_0$, $C([0,1])$, $\ell_p$ and $L_p$, $1 \leqslant p<\infty$, provided it fulfills the necessary condition of having a non-trivial central involution.
Therefore, one cannot distinguish these spaces by the countable discrete or even just closed subgroups of $S_\infty$ displayable on them.

We shall also address and solve the question of the displayability of two well-known uncountable groups, namely, the group of isometries and the group of rotations of the $2$-dimensional Hilbert space. 
Moreover, in a general direction, we show that for any Polish group $G$, the group $\{-1,1\} \times G$ is displayable on some separable real Banach space explicitly constructed from $G$.

Finally, we shall make some observations about the relation between so-called LUR renormings, a crucial tool for our main result, and transitivity and maximality of norms.

\section{Displays: necessary conditions}

In what follows,
the closed unit ball of a Banach space $X$ will be denoted by $B_X$ and the unit sphere by $S_X$.
The space of continuous linear operators on $X$ will be denoted by $L(X)$ and the group of linear automorphisms on $X$ by $GL(X)$.
Furthermore, the group ${\rm Isom}(X)$ of surjective linear isometries on a Banach space $X$ is a closed subset of $GL(X)$ for the strong operator topology and, when $X$ is separable, the unit ball of $(L(X),{\norm{\cdot}})$ is a Polish space for \textsc{sot}. In this case, it follows that the group of isometries on $X$ is a Polish topological group with respect to \textsc{sot}.

Let now $X$ be a  Banach space and assume $G$ is a group of isometries or automorphisms on $X$. We wish to determine whether $X$ admits an equivalent norm for which the group of isometries is $G$, or, in other words, whether the canonical injection of $G$ into $GL(X)$ is a display.  If $G$ is a bounded group of automorphisms, meaning that $\sup_{T\in G}\norm{T}<\infty$, it is always possible to renorm $X$ to make these automorphisms into isometries. Namely,  just define the new norm by $\triple{x}=\sup_{T \in G}\|Tx\|$. So, in what follows, we shall often assume that $G$ is represented as a group of isometries on $X$ and then ask whether this representation is actually a display of $G$ on $X$. Two necessary conditions for this were considered in \cite{FG}, and we will here use the concept of convex transitivity, inspired by a definition of A. Pe\l czy\'nski and  S. Rolewicz \cite{R}, to add a third necessary condition.

\begin{defi} Let $(X,{\norm{\cdot}})$ be a real Banach space and $G$ be a subgroup of ${\rm Isom}(X)$. We shall say that $G$ acts convex transitively on $X$ if, for any $x \in S_X$, the closed unit ball of $X$ is the closed convex hull of the orbit $Gx$. 
\end{defi}

We have the following elementary reformulation.

\begin{lemme} \label{conv trans}
Let $(X,{\norm{\cdot}})$ be a real Banach space and $G$ be a subgroup of ${\rm Isom}(X)$. Then $G$ acts convex transitively on $X$ if and only if for all $x \in S_X$ and $x^* \in S_{X^*}$,
$$
\sup_{T \in G}\;x^*(Tx) =1.
$$
\end{lemme}

\begin{proof} Note that, if $G$ does not act convex transitively on $X$, then we may find $x,y \in S_X$ such that $y \notin \overline{\rm conv}\ Gx$. Applying the Hahn-Banach Theorem, we then obtain a normalized  functional $x^*$ such that $\sup_{T \in G} x^*(Tx) < x^*(y) \leqslant 1$. Conversely, if $\sup_{T \in G}x^*(Tx)<1$ for some normalized $x\in X$ and $x^*\in X^*$, then
$\sup\;\{x^*(y):  y \in {\rm conv}\ Gx\}=\delta<1$, and therefore, if $y$ is some normalized vector such that $x^*(y)>\delta$, then $y$ does not belong to the closed convex hull of $Gx$. So the action of $G$ is not convex transitive.
\end{proof}

A typical example of a proper subgroup $G$ of ${\rm Isom}(X)$, that acts convex transitively (and even transitively) on $X$, is the group of rotations in the $2$-dimensional euclidean space.

\begin{prop}[Necessary conditions]\label{necessary} Let $X$ be a real Banach space and $G$ be a proper subgroup of ${\rm Isom}(X,\norm\cdot)$, which is the group of isometries on $X$ in some equivalent norm $\triple\cdot$. Then
\begin{itemize}
\item[(i)] $G$ contains $-\Id$,
\item[(ii)] $G$ is closed,
\item[(iii)] there exist $x \in S^{\norm\cdot}_X$, $x^* \in S^{\norm\cdot}_{X^*}$ such that $\sup_{T \in G}x^*(Tx)<1$.
\end{itemize}
\end{prop}

\begin{proof} The first two conditions are obvious. For the third condition, we shall imitate E. Cowie's proof \cite{C} to the effect that convex transitive norms are uniquely maximal. So suppose towards a contradiction that (iii) fails, which by Lemma \ref{conv trans} means that $G$ acts convex transitively on $(X,\norm\cdot)$, and assume, without loss of generality, that $\triple{\cdot} \leqslant {\norm{\cdot}}$.
If $x$ has ${\norm{\cdot}}$-norm $1$, then by convex transitivity, 
$$
\overline{\rm conv}\ Gx=B^{\norm\cdot}_X.
$$
Given $y \in S^{\norm\cdot}_X$ and $\epsilon>0$, there exists $U_1,\ldots,U_n$ in $G$ and a convex combination $\sum_i \lambda_i U_i$ such that
$$
\Norm{y-\sum_i \lambda_i U_i(x)} \leqslant \epsilon,
$$
and hence 
$$
\triple{y} \leqslant \Triple{y-\sum_i \lambda_i U_i(x)}+\Triple{\sum_i \lambda_i U_i(x)}
$$
$$
\leqslant \Norm{y-\sum_i \lambda_i U_i(x)}+\sum_i \lambda_i \Triple{U_i(x)} \leqslant \epsilon + \triple{x},
$$
since each $U_i$ is also a $\triple{\cdot}$-isometry.
As $\epsilon>0$ was arbitrary, we deduce that $\triple{y}\leqslant \triple{x}$ and by symmetry that $\triple{y}=\triple{x}$.
Therefore $\triple{x}$ is constant on the unit sphere $S_X^{\norm\cdot}$ of $X$, which means that $\triple{\cdot}$ is a multiple of ${\norm{\cdot}}$.
In particular, 
$$
G={\rm Isom}(X,\triple{\cdot})={\rm Isom}(X,{\norm{\cdot}}),
$$
contradicting that $G$ is a proper subgroup of ${\rm Isom}(X,{\norm{\cdot}})$.
\end{proof}

It would certainly be optimistic to hope these necessary conditions to be sufficient in general. 
However, instead, we may try to diminish the gap between our necessary and sufficient conditions and shall therefore now turn our attention to sufficient conditions. 

\begin{defi} 
Let $X$ be a Banach space, $G$ a group of automorphisms on $X$ and $x\in X$. We shall say that $x$ is {\em distinguished} by $G$ if $\inf_{T \in G\setminus \{\Id\}}\norm{Tx-x}>0$.
\end{defi}

In other words, $x$ is distinguished by $G$ if the orbit map $T\in G\mapsto Tx\in X$ is injective and, moreover, the orbit $Gx$ is discrete.
It follows that if $G$ is an \textsc{sot}-closed group of isometries with a distinguished point, then $G$ must be \textsc{sot}-discrete, and hence countable whenever $X$ is separable. 

\begin{thm}[V. Ferenczi and E. M.  Galego \cite{FG}]\label{FG} Let $X$ be a separable real Banach space with an LUR norm. Let $G$ be a  group of isometries on $X$ which contains $-\Id$ and admits a distinguished point. Then there exists an equivalent norm ${\triple{\cdot}}$ on $X$ such that $G={\rm Isom}(X,{\triple{\cdot}})$.
\end{thm}

Here we recall that a norm $\norm{\cdot}$ is {\em locally uniformly rotund} or {\em LUR} at a point $x\in S_X$ if the following condition holds
$$
\a \epsilon>0 \;\; \e\delta>0\;\; \a y\in S_X\; \big ( \norm{x-y} \geqslant \epsilon\saa \norm{x+y} \leqslant 2-\delta\big).
$$
Equivalently, the norm is LUR at $x$ if $\lim_n x_n=x$ whenever $\lim_n \norm{x_n}=\norm{x}$ and $\lim_n \norm{x+x_n}=2\norm{x}$.
Also, the norm of $X$ is said to be {\em locally uniformly rotund} if it is locally uniformly rotund at every point of $S_X$.
It may  be seen directly that a group  of isometries  on a  space with LUR norm and with a distinguished point must satisfy the last necessary condition of Proposition \ref{necessary}. In other words, we have the following.

\begin{obs} 
Let $X$ be a real Banach space, $y$ be a point of $S_X$ where the norm is LUR and let $G\leqslant {\rm Isom}(X)$ be such that the orbit $Gy$ is discrete. Then there exist $x \in S_X$, $x^* \in S_{X^*}$ such that $\sup_{T \in G}x^*(Tx)<1$.
\end{obs} 

\begin{proof} Let $\alpha>0$ be such that $\|y-Ty\| \geqslant \alpha$ whenever $ Ty \neq y$. Since the norm is LUR at $y$, it follows that there is some $\epsilon>0$ such that $\|y+Ty\| \leqslant 2-2\epsilon$ whenever $Ty \neq y$. Moreover,  we may assume that $2\epsilon<\alpha$.
Pick some $x \in S_X$ such that $0<\norm{x-y}<\epsilon$, and let $x^* \in S_{X^*}$ be such that $x^*(y)=1$. By the LUR property in $y$, let $\beta>0$ be such
that 
$$
\norm{y-z} \geqslant \norm{y-x} \Rightarrow \|z+y\| \leqslant 2-\beta.
$$
Now, if for some $T \in G$, $Ty=y$, then $\norm{y-Tx}=\norm{Ty-Tx} = \norm{y-x}$. 
If, on the other hand, $Ty \neq y$, then 
$$
\norm{y-Tx} \geqslant  \norm{Ty-y}-\norm{Ty-Tx} \geqslant \alpha-\norm{y-x} \geqslant \norm{y-x}.
$$
In either case, $\norm{y-Tx} \geqslant  \norm{y-x}$ and therefore $\|Tx+y\| \leqslant 2-\beta$ and
$$
x^*(Tx)=x^*(Tx+y)-x^*(y) \leqslant 1-\beta.
$$
As this holds for any $T \in G$, this completes the proof.
\end{proof}

This result suggests the existence of positive results about displays of groups under hypotheses of existence of discrete orbits, weaker than the existence of a distinguished point.  So, in what follows, we shall consider a closed subgroup $G$ of $S_\infty$ represented as a group of isometries on a Banach space $X$ and characterize when such a representation is a display. We first  consider the case when $X=c_0$ and shall see that there are explicit definitions of renormings on $c_0$ which transform a representation into a display.

%%%%%%%%%%%%%%%%%%%%%%%%%%%%%%%%%%%%%%%%%%%%

\section{Subgroups of $S_{\infty}$ and the space $c_0$}

We recall that the {\em infinite symmetric group}, $S_\infty$, is the group of all permutations of the natural numbers $\N=\{0,1,2,\ldots\}$, which becomes a Polish group when equipped with the {\em permutation group topology} whose basic open sets are of the form 
$$
V(f,x_1,\ldots,x_n)=\{g\in S_\infty\del g(x_i)=f(x_i) \text{ for all $i=1,\ldots, n$}\},
$$
where $f\in S_\infty$ and $x_1,\ldots, x_n\in \N$.

\subsection{Groups of automorphism on graphs}

We shall need the following folklore result.
\begin{prop}\label{graph}
Any closed subgroup of $S_\infty$ is topologically isomorphic to the full automorphism group of a countable connected graph.
\end{prop}

\begin{proof}
Suppose $G\leqslant S_\infty$ is a closed subgroup. The action of $G$ on $\N$ extends canonically to a continuous action of $G$ on the countable discrete set $\N^{<\N}$ of all finite sequences of natural numbers by
$$
g\cdot (x_1,\ldots,x_n)=(g(x_1),\ldots, g(x_n)).
$$
Let also $X=\N^0\cup \N^1\cup \N^2\cup \N^4\cup \N^8\cup\ldots$ and let $o\colon X\til \{n\in \N\del n\geqslant 7\}$ be a function satisfying
$$
o(s)=o(t)\;\equi\; \e g\in G\;\;\; g\cdot s=t.
$$
We can now define a countable graph $\Gamma$ with vertex set $V(\Gamma)$ and edge set $E(\Gamma)$ as follows. 

Using $X$ as an initial set of vertices, we define $\Gamma$ as follows:
\begin{itemize}
\item If $s=x_1\ldots x_{n}\in X$ and $t=y_1\ldots y_{n}\in X$, $n\geqslant 1$, we connect $s$ to $st$ as follows.

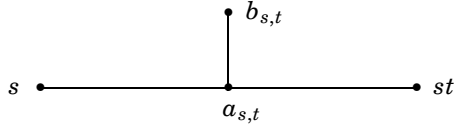
\setlength{\unitlength}{.5cm}
\begin{picture}(10,5)(0,-1)
\put(0,0){$\bullet$}
\put(5,0){$\bullet$}
\put(5,2){$\bullet$}
\put(10,0){$\bullet$}

\put(0.2,0.15){\line(1,0){10}}
\put(5.15,0.1){\line(0,1){2}}
\put(-.7,0){$s$}
\put(10.6,0){$st$}
\put(5,-.6){$a_{s,t}$}
\put(5.6,2){$b_{s,t}$}
\end{picture}

\noindent
and connect $t$ to $st$ as follows

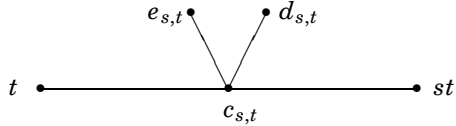
\begin{picture}(10,5)(0,-1)
\put(0,0){$\bullet$}
\put(5,0){$\bullet$}
\put(4,2){$\bullet$}
\put(6,2){$\bullet$}
\put(10,0){$\bullet$}

\put(0.2,0.15){\line(1,0){10}}
\put(5.15,0.15){\line(1,2){1}}
\put(5.15,0.15){\line(-1,2){1}}
\put(-.7,0){$t$}
\put(10.6,0){$st$}
\put(5,-.6){$c_{s,t}$}
\put(6.5,2){$d_{s,t}$}
\put(3,2){$e_{s,t}$}
\end{picture}

\item Also, we connect any $x\in \N^1$ to $\tom\in \N^0$ by a simple edge

\begin{picture}(10,3)(0,-1)
\put(0,0){$\bullet$}
\put(5,0){$\bullet$}

\put(0.2,0.15){\line(1,0){5}}

\put(-.7,0){$x$}
\put(5.6,0){$\tom$}
\end{picture}

\item Moreover, for any $s=x_1\ldots x_{n}\in X$, $n\geqslant 0$, we add $o(s)$ new neighbouring vertices

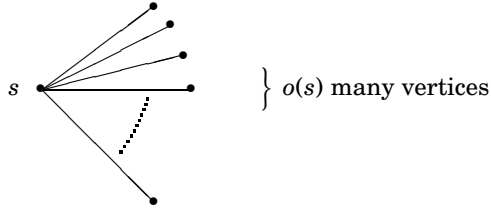
\begin{picture}(10,7)(0,-4)
\put(0,0){$\bullet$}
\put(4,0){$\bullet$}
\put(3.8,.9){$\bullet$}
\put(3.45,1.7){$\bullet$}
\put(3,2.2){$\bullet$}
\put(3,-3){$\bullet$}

\put(0.2,0.15){\line(4,3){3}}
\put(0.2,0.15){\line(2,1){3.4}}
\put(0.2,0.15){\line(4,1){3.7}}
\put(0.2,0.15){\line(1,0){4}}

\put(0.2,0.15){\line(1,-1){3}}

\put(-.7,0){$s$}
\put(6,0){$\Big\}\;\;o(s) \text{ many vertices}$}

\thicklines

\qbezier[10](3,-.1)(2.8,-.8)(2.3,-1.5)
\end{picture}

\end{itemize}
This ends the description of $\Gamma$. We now see that every vertex in $X$ has infinite valence, while every vertex in $V(\Gamma)\setminus X$ has finite valence. Thus, $X$ is invariant under automorphisms of $\Gamma$. Moreover, every vertex $s\in X$ neighbours exactly $o(s)$ vertices of valence $1$, which shows that also $o(\cdot)$ is invariant under automorphisms of $\Gamma$.

Now, suppose $g$ is an automorphism of $\Gamma$. Then $g$ induces a permutation of $X$ preserving $o(\cdot)$ and, in particular, $g[\N^n]=\N^n$ for all even $n$. We claim that if $s,t,u,v\in \N^n$ and $g(st)=uv$, then also $g(s)=u$ and $g(t)=v$. But this is easy to see since $s$ and $t$ are the unique vertices connected to $st$ by the schema above and similarly for $u$, $v$ and $uv$.

It follows from this that $g(x_1,\ldots,x_n)=(g(x_1),\ldots, g(x_n))$ for any $(x_1,\ldots, x_n)\in X$. And so $g$ as a permutation on $X$ satisfies 
$$
g(s)=t\;\saa \; o(s)=o(t)\;\saa\; \e f\in G\; f\cdot s=t
$$
for any $s,t\in X$. Since $G$ is closed in $S_\infty$ it follows that the induced permutation $g$ of $\N$ belongs to $G$. Conversely, any element of $G$ easily induces an automorphism of $\Gamma$. Furthermore, it is clear that the corresponding map is a topological isomorphism.
\end{proof}

%Observe that in this construction the point $\emptyset$ of the graph $\Gamma$ is a fixed point %of any automorphism of $\Gamma$. So in Proposition we may assume the existence of a point %of $\Gamma$ which is invariant under any automorphism. 

\subsection{New norms on $c_0$ }
Suppose $\Gamma$ is a connected graph on $\N$ with corresponding path metric $d_\Gamma$ and let $c_{00}$ denote the vector space with basis $(e_n)_{n\in \N}$. We define a norm $\|\cdot\|_\Gamma$ on $c_{00}$ by
\begin{align}
\|\sum a_ne_n\|_\Gamma=\;&\sup \Big\{\big|a_n+\frac{a_m}{1+2d_\Gamma(n,m)}\big|\colon  n\neq m\Big\}\\
\bigvee \;&\sup \Big\{\big|a_n-\frac{a_m}{2+2d_\Gamma(n,m)}\big|\colon  n\neq m\Big\}\\
\bigvee\; & \|(a_n)\|_\infty.
\end{align}
Then clearly $\|\cdot\|_\infty\leqslant \|\cdot\|_\Gamma\leqslant \frac 43\|\cdot\|_\infty$, and hence the completion of $c_{00}$ with respect to the norm $\|\cdot\|_\Gamma$ is just $c_0$.  Note that for any $\sum a_ne_n\in c_0$, we have 
$$
\|\sum a_ne_n\|_\Gamma=\|\sum a_ne_n\|_\infty
$$
if and only if $\sum a_n e_n=a_p e_p$ for some $p\in \N$.

We denote by $B(c_0, \|\cdot\|_\Gamma)$ the closed unit ball of $(c_0, \|\cdot\|_\Gamma)$.
\begin{lemme}
For every $p\in \N$, $e_p$ is an extreme point of $B(c_0, \|\cdot\|_\Gamma)$.
\end{lemme}

\begin{proof}
Note that if $0<\lambda<1$, $\sum a_ne_n,\sum b_ne_n\in B(c_0,\|\cdot\|_\Gamma)$ and 
$$
e_p=\lambda \sum a_ne_n+(1-\lambda)\sum b_n e_n=\sum \big(\lambda a_n+(1-\lambda)b_n\big)e_n,
$$
then, as $|a_n|,|b_n|\leqslant 1$ for all $n$, we have $a_p=b_p=1$ and so $\|\sum a_n e_n\|_\infty=1=\|\sum a_ne_n\|_\Gamma$ and $\|\sum b_n e_n\|_\infty=1=\|\sum b_ne_n\|_\Gamma$. It thus follows that $a_n=b_n=0$ for all $n\neq p$, i.e., $\sum a_ne_n=\sum b_n e_n=e_p$, showing that $e_p$ is an extreme point.
\end{proof}

\begin{lemme}
If $\sum a_ne_n\in B(c_0,\|\cdot \|_\Gamma)$ is an extreme point, then $\sum a_n e_n=\pm e_p$ for some $p\in \N$.
\end{lemme}

\begin{proof}By the comments above, it suffices to show that $\|\sum a_ne_n\|_\infty=\|\sum a_n e_n\|_\Gamma$.
So suppose towards a contradiction that $\|\sum a_n e_n\|_\infty <1$ and find some $0<\delta<\frac 12$ such that $\|\sum a_ne_n\|_\infty<1-\delta$. Pick also some $p$ such that $|a_p|<\frac \delta 2$. 
Then for all $m\neq p$, 
$$
\Big| a_p\pm \frac \delta2+\frac {a_m}{1+2d_\Gamma(p,m)}\Big|<1,
$$
$$
\Big| a_m+\frac {a_p\pm \frac \delta2}{1+2d_\Gamma(m,p)}\Big|<1.
$$
$$
\Big| a_p\pm \frac \delta2+\frac {a_m}{2+2d_\Gamma(p,m)}\Big|<1,
$$
$$
\Big| a_m+\frac {a_p\pm \frac \delta2}{2+2d_\Gamma(m,p)}\Big|<1.
$$
and
$$
|a_p\pm \frac \delta2|<1,
$$
from which is follows that  $\big\|\sum a_ne_n\pm \frac \delta2a_p\big\|=1$. But 
$$
\sum a_ne_n=\frac 12\big(\sum a_ne_n+ \frac \delta2a_p\big )+\frac 12\big(\sum a_ne_n-\frac \delta2a_p\big),
$$
contradicting that $\sum a_ne_n$ is an extreme point.
\end{proof}
So the extreme points of $B(c_0,\|\cdot \|_\Gamma)$ are just $\pm e_p$ for $p\in \N$. 

\begin{lemme}\label{lemme15}
If $T\colon (c_0,\|\cdot \|_\Gamma)\til (c_0,\|\cdot \|_\Gamma)$ is a surjective linear isometry, then there is an automorphism $\phi\colon \Gamma\til \Gamma$ and $\eps=\pm 1$ such that
$$
T(e_n)=\eps e_{\phi(n)}
$$
for all $n\in \N$.
\end{lemme}

\begin{proof}
Note that $T$ as well as $T^{-1}$ preserves the extreme points of $B(c_0,\|\cdot \|_\Gamma)$ and so for some permutation $\phi$ of $\N$, we have $T(e_n)=\pm e_{\phi(n)}$. We first show that the choice of sign is uniform for all $n$. To see this, note that for any $n\neq m$, 
$$
\|e_n+e_m\|=1+\frac 1{1+2d_\Gamma(n,m)},
$$
while 
$$
\|e_n-e_m\|=1+\frac1{ 2+2d_\Gamma(n,m)}.
$$
So by considering the parity of the denominator in these fractions, we see that for all $n\neq m$ and $p\neq q$,
$$
\|e_n+e_m\|\neq \|e_p-e_q\|,
$$
whence if $T(e_n)=e_{\phi(n)}$, also $T(e_m)=e_{\phi(m)}$ for all $m\in \N$. This accounts for the choice of $\eps=\pm 1$. Now, to see that $\phi$ is an automorphism, we simply note that for any $n\neq m$, the distance $d_\Gamma(n,m)$ can be read off from $\|e_n+e_m\|=\|e_{\phi(n)}+e_{\phi(m)}\|$. So $d_\Gamma(n,m)=d_\Gamma(\phi(n),\phi(m))$, and hence $\phi$ preserves the edge relation.
\end{proof}

\begin{thm} Let $G$ be a closed subgroup of $S_{\infty}$. Then there exists  $\Gamma$  a connected graph on $\N$ such that the group of isometries on $(c_0,\|\cdot\|_\Gamma)$ is topologically isomorphic to $\{-1,1\} \times G$. Therefore $\{-1,1\} \times G$ is displayable on $c_0$.
\end{thm}

\begin{proof}
Use Proposition \ref{graph} to pick $\Gamma$ such that ${\rm Aut}(\Gamma)$ is topologically isomorphic to $G$. Then apply Lemma \ref{lemme15}.
\end{proof}

%\begin{thm}
%Let $H$ be a closed subgroup of $S_\infty$. Then $\{-1,1\} \times H$ is displayable on the %separable Hilbert space.
%\end{thm}

%\begin{proof}
%By Proposition \ref{graph} and Theorem \ref{c0}, the group of isometries on $(c_0,%\norm{\cdot}_{\Gamma})$ is topologically isomorphic to $\{-1,1\} \times Aut(\Gamma) \simeq %\{-1,1\} \times H$. Furthermore  from the fact that any automorphism of $\Gamma$ maps %$\emptyset$ to itself,  it follows that there is some $n$, say $n=0$, such that any isometry on %$(c_0,\norm{\cdot}_{\Gamma})$ maps $e_n$ to $\pm e_n$.

%Consider $X=\ell_2$ with the euclidean norm ${\norm{\cdot}}$. Then $Y=X \oplus \R$ is again %isomorphic to $\ell_2$. Let $p(.)=\norm{\cdot}_{\Gamma} $ and $x_0=e_0$. Note that $G_2$ in %Lemma \ref{jarosz2}  is  just the group of isometries on $(\ell_2,\norm{\cdot}_{\Gamma})$, which is %easily topologically isomorphic to the group of isometries on $(c_0,\norm{\cdot}_{\Gamma})$ and %therefore to $\{-1,1\} \times H$. On the other hand it is clear that $G_2 \subseteq G_1$, therefore %Lemma \ref{jarosz2} implies that $\{-1,1\} \times H$ is topologically isomorphic to the group %of isometries on $(\ell_2,{\norm{\cdot}}_w)$.
%\end{proof} 

%%%%%%%%%%%%%%%%%%%%%%%%%%%%%%%%%%%%%%

\section{Renormings of separable Banach spaces}

This section is devoted to displaying subgroups of $S_\infty$ on  spaces with a symmetric basis, such as $\ell_p, 1 \leqslant p<+\infty$, or with a symmetric decomposition, such as $L_p, 1 \leqslant p<+\infty$ or $C([0,1])$.
In order to do this, we shall extend Theorem \ref{FG} by proving that any representation of a group $G$ on a separable real space $X$ with an LUR norm,  as a closed group of isometries containing $-\Id$ and with sufficiently many  discrete orbits,  is actually a display.
We first recall a definition from \cite{FG} which is an extension of ideas and terminology of Bellenot \cite{Bel}.

 \subsection{Bellenot's renormings}
\begin{defi} Let $X$ be a real Banach space with norm ${\norm{\cdot}}$,  $G$ be a  group of isometries on $X$ containing $-\Id$ and let $(x_k)_{k \in K}$ be a possibly finite sequence of normalized vectors of $X$. Let
$\Lambda=(\lambda_k)_{k \in K}$ be such that $1/2<\lambda_k<1$ for all
$k \in K$. The {\em $\Lambda,G$-pimple norm at $(x_k)_k$ for ${\norm{\cdot}}$} is the equivalent norm on $X$ defined by
$$
\norm{y}_{\Lambda,G}=\inf\{\sum [[y_i]]_{\Lambda,G}\del  y=\sum y_i\},
$$
where $[[y]]_{\Lambda,G}=\lambda_k \norm{y},$
whenever $y \in \R\cdot gx_k$ for some $k \in K$ and $g \in G$, and
$[[y]]_{\Lambda,G}=\norm{y}$ otherwise.
\end{defi} 

In other words, the closed unit ball for $\norm{\cdot}_{\Lambda,G}$ is the closure of the convexification of the union of the closed unit ball for $\norm{\cdot}$ with line segments between $gx_k/\lambda_k$ and $-gx_k/\lambda_k$ for each $k \in K$ and $g \in G$.

When there is just one point $x_0$ associated to $\lambda_0$ and if $G=\{\Id\}$, then we just add two spikes to the ball on $x_0$ and $-x_0$. This is the original $\lambda_0$-pimple norm $\norm{\cdot}_{\lambda_0,x_0}$ at $x_0$ defined by S. Bellenot in \cite{Bel}. In the following, we shall denote the $\lambda_k$-pimple norm at $gx_k$ by $\norm{\cdot}_{\lambda_k,g}$ and relate the properties of the norm ${\norm{\cdot}}_{\Lambda,G}$ to the properties of each of the norms $\norm{\cdot}_{\lambda_k,g}$.

It may be observed that
$$(\inf_{k \in K}\lambda_k)\norm{\cdot} \leqslant \norm{\cdot}_{\Lambda,G} \leqslant \norm{\cdot},$$
so that $\norm{\cdot}_{\Lambda,G}$ is an equivalent norm on $X$, and that by
definition, 
any $g \in G$ remains an isometry in the norm $\norm{\cdot}_{\Lambda,G}$.

 As in \cite{Bel} we shall say that a normalized vector $y$ is {\em extremal} for a
norm $\norm{\cdot}$ if it is an extremal point of the closed unit ball for
$\norm{\cdot}$,  that is, if
whenever $\norm{y}=\norm{z}=1$ and $x$ belongs to the segment $[y,z]$, then
$y=x=z$. A norm is {\em strictly convex} at a point $x$ of the unit sphere
if $x$ is  extremal.  Note that, if  a norm is LUR at $x$, then it must be strictly convex at $x$.

 We recall a key result from \cite{Bel} (Proposition p. 90).

\begin{prop}\label{bellenot}(S. Bellenot \cite{Bel}) Let $(X,\norm{\cdot})$ be a real
  Banach space and let $x_0 \in X$ be normalized such that
\begin{itemize}
\item[(1)] $\norm{\cdot}$ is LUR at $x_0$, and 
\item[(2)] there exists $\epsilon>0$ such that if $\norm{y}=1$ and
$\norm{x_0-y} < \epsilon$, then ${\norm{\cdot}}$ is strictly convex in $y$.
\end{itemize}
Then given $\delta>0$, $b>0$ and $0<m<1$, there exists
a real number $1/2<\lambda_0<1$ such that whenever $\lambda_0 \leqslant \lambda <1$ and
$\norm{\cdot}_{\lambda}$ is the $\lambda$-pimple norm at $x_0$, then 
\begin{itemize}
\item [(3)] $m\norm{\cdot} \leqslant \norm{\cdot}_{\lambda} \leqslant \norm{\cdot}$,
\item [(4)] if $1=\norm{y}>\norm{y}_{\lambda}$ then $\norm{x_0-y} <\delta$
or $\norm{x_0+y}<\delta$,
\item [(5)] $x_{\lambda}=\lambda^{-1}x_0$ is the only isolated extremal point $x$  for $\norm{\cdot}_{\lambda}$ satisfying $\Norm{\frac x{\norm{x}}-x_0} <\epsilon$,
\item [(6)] if $w$ is a vector so that $x_\lambda$ and $x_\lambda+w$ are endpoints of a maximal line segment in the unit sphere of $\norm{\cdot}_\lambda$, then $b \geqslant \norm{w} \geqslant \lambda^{-1}-1$.
\end{itemize}
\end{prop}

 %It may be noted  that the expression of
%$\lambda_0$ in \cite{Bel} makes use of the value of the LUR function $\lambda(x,\eta)$ in 
%$x_0$ for some specific value of $\eta$. For more details we refer
%to \cite{Bel}.

 Our objective is to generalize this result to $(\Lambda,G)$-norms. This was done in \cite{FG} when $G$ was countable and we shall observe here that the result may be extended to certain uncountable groups.

In the following, we let $B$ denote the  closed unit ball for $\norm{\cdot}$,
$B_\Lambda^G$ denote the closed unit ball for the $\Lambda,G$-norm at
$(x_k)_k$, $B_k^g$ denote the closed unit ball for the $\lambda_k$-norm $\norm{\cdot}_{\lambda_k,g}$ at $gx_k$ and
set 
$$
B_0=\bigcup_{k\in K}\bigcup_{g\in G}B_{k}^g.
$$ 
Clearly, $B \subseteq B_0 \subseteq B_\Lambda^G$. 

 If $\Lambda=(\lambda_k)_{k \in K}$,
$\Lambda'=(\lambda_k^{\prime})_{k \in K}$, and $c,d \in \R$, we
  write
$\Lambda < \Lambda'$ to mean $\lambda_k < \lambda_k^{\prime}$ for all
$k \in K$, $c < \Lambda$ to mean $c < \lambda_k$ for all $k \in K$,
and $\Lambda <d$ to mean $\lambda_k<d$ for all $k \in K$. 

We shall first show in Lemma \ref{LemmaGbellenot} that for  $\Lambda$ close enough to $1$,
 $B_\Lambda^G$ is actually equal to $B_0$.
This implies that each spike associated to each point of the form $gx_k$ is sufficiently small and separated from other spikes for us  to be able to apply the results of Proposition \ref{bellenot} independently for each point.

\begin{lemme}\label{LemmaGbellenot}
Let $(X,\norm{\cdot})$ be a real Banach space, $-\Id\in G\leqslant {\rm Isom}(X,\norm\cdot)$ and let $(x_k)_{k \in K}$
be a  finite or infinite sequence of normalized vectors of $X$. Assume also that the following conditions hold.
\begin{itemize}
\item[(1)] For each $k \in K$ there is $\eps_k>0$ such that $\norm{\cdot}$ is LUR at $x_k$ and strictly convex in any $y$ such that $\norm{y-x_k} \leqslant \epsilon_k$. 
\item[(2)] For each $k \in K$ there exists $c_k>0$ such that for all $g \in G$ and  $j \in K$, either $\norm{x_j-gx_k} \geqslant c_k$ or $x_j=gx_k$.
\item[(3)] For each $k \in K$, $x_k \notin \bigcup_{j<k}Gx_j$.
\end{itemize}
Then  there exists
$\Lambda_0=(\lambda_{0k})_k$, with $1/2 < \Lambda_0$
 such that whenever $\Lambda=(\lambda_k)_k$ satisfies
 $\Lambda_0 < \Lambda < 1$, it follows that

\begin{itemize}
\item[(a)] whenever $x \in B_k^g \setminus B$ and $y \in B_l^h \setminus B$
  with
$gx_k \neq hx_l$, then $\norm{x-y} \geqslant c_{\min(k,l)}/3$,

\item[(b)] $B_\Lambda^G=B_0$,

\item[(c)] 
$\norm{\cdot}_{\Lambda,G}=\inf_{k \in K, g \in G}\norm{\cdot}_{\lambda_k,g}$.

\item[(d)] whenever $\norm{x}_{\Lambda,G}<\norm{x}$, there exists $(k,g)$ such that
$\norm{x}_{\Lambda,G}=\norm{x}_{\lambda_k,g}$, $k$ and the pair $\{gx_k,-gx_k\}$ are uniquely determined by this property, and for any $(l,h)$ such that $gx_k \neq \pm hx_l$, 
$\norm{x}_{\lambda_l,h}=\norm{x}$.

\end{itemize}

Furthermore for each $k$, $\lambda_{0k}$  depends only on 
$x_i, c_i, \epsilon_i$, $1 \leqslant i \leqslant k$.

\end{lemme}

\begin{proof}
This result was proved in \cite{FG}, under some slightly stronger assumption on the convexity of the norm, and under the stronger condition that $\norm{x_j-gx_k} \geqslant c_k$ unless $j=k$ and $g= \Id$, from which  it followed  that $B_k^g=B_l^h$ if and only if $k=l$ and $g=\pm h$.  It was then proved that the spikes associated to the points $gx_k$ were sufficiently distant from each other. In the proof, the condition $(k,g) \neq (l,\pm h)$ is sometimes considered, but what really matters for is that $B_k^g \neq B_l^h$.
Observe that the two balls $B_k^g$ and $B_l^h$ are equal if and only if $gx_k=\pm hx_l$ (which implies that $k=l$ by condition (3)).

 Note that by (1), $\norm{\cdot}$ is LUR at $x_k$ for each $k$, so Proposition
\ref{bellenot} (1) is satisfied for $x_0=x_k$. By (1) again,
$\norm{\cdot}$ is strictly convex in a neighborhood of $x_k$, so  Proposition \ref{bellenot}
(2) applies in $(x_k)$ for $\epsilon=\epsilon_k$.
 We may assume that
  $\epsilon_k \leqslant c_k/2$, and fix a decreasing sequence  $(\delta_k)_k$ such
  that for all $k \geqslant 1$, $\delta_k \leqslant c_k/4$ and
$\frac{4\delta_k}{3} \leqslant 1-\lambda(x_k,c_k)$, where $\lambda(x_k,\cdot)$ is the LUR function
associated to the norm $\norm{\cdot}$ in $x_k$, i.e.,
$$
\a \epsilon>0\;\; \a y\in S_X\; \big ( \norm{x_k-y} \geqslant \epsilon\saa \NORM{\frac{x_k+y}2} \leqslant \lambda(x_k,\eps)\big).
$$

For each $k$, we let
 $\lambda_{0k}$ be the real $\lambda_0$ associated by Proposition \ref{bellenot} to $x_0=x_k$, 
$\epsilon=\epsilon_k$, $\delta=\delta_k$, $b=1$ and $m=1/2$.
Up to replacing each $\lambda_{0k}$ by a larger number in $]1/2,1[$, we may
 assume 
that  $\lambda_{0k}^{-1}-1 \leqslant \delta_k/3$ for all
$k \in K$ and that $\lim_{k \rightarrow \infty} \lambda_{0k}=1$ if $K$ is
infinite.

We let $\Lambda=(\lambda_k)_k$ be such that $\Lambda_0 < \Lambda
< 1$.

We first prove (a).
 Whenever
 $x \in B_k^g \setminus B$ 
we have, by Proposition \ref{bellenot} (4), that
$\norm{z-gx_k} < \delta_k$ or $\norm{z+gx_k} <\delta_k$, where $z=x/\norm{x}$.
 Up to redefining $g$
 as $-g$ if necessary we may assume that the first holds.
Then
$$
\norm{x-gx_k} \leqslant \norm{z-x}+\norm{z-gx_k} <\norm{x}-1+\delta_k
\leqslant \lambda_k^{-1}-1+\delta_k \leqslant \frac{4\delta_k}{3}.
$$
Likewise if $y \in B_l^h \setminus B$ then up to redefining $h$ as $-h$,
$$
\norm{y-hx_l} <\frac{4\delta_l}{3}.
$$
If now $gx_k \neq hx_l$ and say $k \leqslant l$, we have,
that 
$$
\norm{x-y} \geqslant \norm{gx_k- hx_l}-\norm{x-gx_k}-\norm{y-hx_l} \geqslant
c_k-\frac{4}{3\delta_k}-\frac{4}{3\delta_l},
$$
so since $\delta_l \leqslant \delta_k$,
$$
\norm{x-y}  \geqslant c_k-\frac{8\delta_k}{3} \geqslant c_k/3.
$$
Therefore, (a) is proved.

We shall now prove (b).
First we observe that $B_0$ is closed. Indeed, suppose that  $x$ is the limit of a
convergent sequence $(x_n)$ in $B_0$, and let $k_n, g_n$ be such that
$x_n \in B_{k_n}^{g_n}$. We claim that $x \in B_0$. For if $x_n \in B$ for infinitely many $n$, then $x \in B$
and we are done, so we may assume that $x_n \in B_{k_n}^{g_n} \setminus B$ for
each $n$. If $k_n$ is bounded, then we can assume that $k_n$ is constantly
equal to some $k$ and that $\norm{x_m-x_n} <c_k/3$ for all $n, m$. But then by
(a), $B_k^{g_n}=B_k^{g_m}$ for all $n,m$, so $x$ also belongs to $B_k^{g_n}$
for any choice of $n$ and therefore to $B_0$, and we are done. So we
may assume that $k_n$  converges to $\infty$. Since $\lambda_{0 k_n}$
converges to $1$, $\lambda_{k_n}$ also converges to $1$, and since $\norm{x_n}
\leqslant 1/\lambda_{k_n}$ for each $n$, $\norm{x} \leqslant 1$. Therefore, $x \in B
\subseteq B_0$, which proves the claim.  Finally, $B_0$ is closed.

Next we observe that $B_0$ is convex.
Assuming towards a contradiction that $x,y \in B_0$ and $\frac{x+y}{2} \notin
B_0$, let
$(k,g)$ and $(l,h)$ be such that $x \in B_k^g$ and $y \in B_l^h$, and without
loss of generality assume that $k \leqslant l$. By convexity
of $B_k^g$ and $B_l^h$, these two balls are different, otherwise
$\frac{x+y}{2}$ would belong to either of them and therefore to $B_0$. This means that $gx_k \neq \pm hx_l$.
% either $k \neq l$ (and we may assume that for example $k<l$),
% or $k=l$
%and $g \neq \pm h$,
 Furthermore, $x \in B_k^g \setminus B$, otherwise $x \in B \subseteq B_l^h$ and
$\frac{x+y}{2} \in B_l^h \subseteq B_0$. 
In other words, 
$\norm{x}_{\lambda_k,g}<\norm{x}$. Likewise,
$\norm{y}_{\lambda_l,h}<\norm{y}$.
Therefore, by Proposition \ref{bellenot} (4) applied to $x/\norm{x}$ for the $\lambda_k$-pimple norm at $gx_k$, and up to replacing $g$ by $-g$ if necessary,
$\Norm{gx_k-\frac x{\norm{x}}} < \delta_k$. 
Then 
$$
\norm{gx_k-x} \leqslant \NORM{gx_k-\frac{x}{\norm{x}}}+\NORM{x-\frac{x}{\norm{x}}} \leqslant
\delta_k+\lambda_k^{-1}-1 \leqslant \frac{4\delta_k}{3}.
$$

Likewise, $\norm{hx_l-y} <\frac{4\delta_l}{3}$ and so 
$$
\NORM{\frac{x+y}{2}-\frac{gx_k+hx_l}{2}}\leqslant \frac{2(\delta_k+\delta_l)}{3} \leqslant \frac{4\delta_k}{3}.
$$
Since $\norm{gx_k-hx_l} \geqslant c_k$ by (2), it follows by the  LUR-property of $\norm{\cdot}$ in $gx_k$ that
$$
\NORM{\frac{gx_k+hx_l}{2}} \leqslant \lambda(gx_k,c_k)=\lambda(x_k,c_k),
$$
and
$$
\NORM{\frac{x+y}{2}} \leqslant \frac{4\delta_k}{3}+\lambda(x_k,c_k)\leqslant 1,
$$
a contradiction,
since $\frac{x+y}{2}$ does not belong to $B_0$ and therefore neither to $B$.
This contradiction proves that $B_0$ is convex.

 Finally we have proved that $B_0$ is closed convex. Since it contains $B$ and each segment
$[-gx_k/\lambda_k,gx_k/\lambda_k]$, it therefore contains $B_{\Lambda}^G$, and
since also $B_0$ is included in $B_{\Lambda}^G$,  it follows that
$B_0=B_{\Lambda}^G$, that is, (b) is proved.

The equality in (c), 
$$
\norm{\cdot}_{\Lambda,G}=\inf_{k \in K,g \in G}\norm{\cdot}_{\lambda_k,g},
$$
follows immediately from (b).

To prove (d), let $x$ be such that $\norm{x}_{\Lambda,G} < \norm{x}$. Then
by (c) there exists $\lambda_k,g$ such that $\norm{x}_{\lambda_k,g}<\norm{x}$.
Therefore, $z=x/\norm{x}_{\lambda_k,g} \in B_k^g \setminus B$.
Since
in  (a), the real $c_{\min(k,l)}$ is  positive, there exists no $B_l^h$ with $hx_l \neq \pm gx_k$ such that $z \in B_l^h \setminus B$. In other words,
$z \notin B_l^h$ for $gx_k \neq \pm hx_l$. 

If we had
that $\norm{x}_{\lambda_l,h}<\norm{x}$ for some $hx_l \neq \pm gx_k$ , then
$z'=x/\norm{x}_{\lambda_l,h}$ would belong to
 $B_l^h \setminus B$. If $\norm{x}_{\lambda_l,h} \leqslant
\norm{x}_{\lambda_k,g}$ then by convexity of $B_{l}^h$, $z \in
B_{l}^h$ and so $z \in B_{l}^h \setminus B$, a contradiction.
If $\norm{x}_{\lambda_l,h} \geqslant
\norm{x}_{\lambda_k,g}$, then we obtain a similar contradiction using $z'$.
Therefore $\norm{x}_{\lambda_l,h} \geqslant \norm{x}$. Finally we have proved that
$\norm{x}_{\lambda_l,h}<\norm{x}$ only if $l=k$ and $hx_l=\pm gx_k$. From (c) we therefore deduce that
$\norm{x}_{\Lambda,G}=\norm{x}_{\lambda_k,g}$.
This concludes the proof of (d).
 \end{proof}

\begin{prop}\label{Gbellenot}
Let $(X,\norm{\cdot})$ be a real Banach space, $-\Id\in G\leqslant {\rm Isom}(X,\norm\cdot)$ and let $(x_k)_{k \in K}$
be a  finite or infinite sequence of normalized vectors of $X$. Assume that the following conditions are satisfied.
\begin{itemize}
\item[(1)] For each $k \in K$ there is $\eps_k>0$ such that $\norm{\cdot}$ is LUR at $x_k$ and strictly convex in any $y$ such that $\norm{y-x_k} \leqslant \epsilon_k$.
\item[(2)] For all $k \in K$ there is $c_k>0$ such that for any $g \in G$ and $j \in K$, either $\norm{x_j-gx_k} \geqslant c_k$ or $x_j=gx_k$.
\item[(3)] For each $k \in K$, $x_k \notin \bigcup_{j<k}Gx_j$.
\end{itemize}
Then given  $(\delta_k)_k>0$, $(b_k)_k>0$ and $0<m<1$, there exists
$\Lambda_0=(\lambda_{0k})_k$, with $1/2 < \Lambda_0$
 such that whenever $\Lambda=(\lambda_k)_k$  satisfies
 $\Lambda_0 < \Lambda < 1$,
we have

\begin{itemize}
\item[(3')] $m\norm{\cdot} \leqslant \norm{\cdot}_{\Lambda,G} \leqslant \norm{\cdot}$,
\item[(4')] if $1=\norm{y}>\norm{y}_{\Lambda,G}$ then  there are  $g \in G$ and $k \in K$ such that $\norm{gx_k-y} <\delta_k$,
\item[(5')] $x_{k,\lambda}=\lambda_k^{-1}x_k$ is the only isolated extremal point $x$ of $\norm{\cdot}_{\Lambda,G}$ satisfying $\Norm{\frac x{\norm{x}}-x_k} <c_k/2$,
\item[(6')] if $w$ is a vector so that $x_{k,\lambda}$ and $x_{k,\lambda}+w$ are endpoints of a maximal line segment in the unit sphere of $\norm{\cdot}_{\Lambda,G}$, then $b_k \geqslant \norm{w} \geqslant \lambda_k^{-1}-1$.
\end{itemize}

Furthermore for each $k$, $\lambda_{0k}$  depends only on 
$m$ and $x_i$, $c_i$, $\delta_i$, $b_i$, $1 \leqslant i \leqslant k$.
\end{prop}

\begin{proof} Once again the proof is not too different from the one in \cite{FG}.

  Fix $G$, $(x_k)_{k \in K}$ and $(\delta_k)_k>0$, $(b_k)_k>0$ and $0<m<1$ as in
 the hypotheses.  We may again assume that $(\delta_k)_k$ is decreasing
 and that for all $k \geqslant 1$, $\delta_k \leqslant c_k/4$, $\delta_k \leqslant \epsilon_k/2$ and
$\frac{3\delta_k}{2} \leqslant 1-\lambda(x_k,c_k)$, and we may also assume that $\epsilon_k \leqslant c_k/2$ and $\delta_k \leqslant \min_{i \leqslant k}c_i/4$.

 Note that as in Lemma \ref{LemmaGbellenot}, by (1),  Proposition
\ref{bellenot} (1) is satisfied for $x_0=x_k$ and Proposition \ref{bellenot} (2) applies in $(x_k)$ for  $\epsilon=\epsilon_k$.

Let therefore, for each $k$,  $\lambda_{0k}$ be the $\lambda_0$ given by Proposition
 \ref{bellenot} for 
 $x_0=x_k$, 
$\epsilon=\epsilon_k$, $\delta=\delta_k$, $b=b_k$ and $m$.
Up to replacing each $\lambda_{0k}$ by a larger number in $]1/2,1[$, we may
 also assume that  (a) to (d) of Lemma \ref{LemmaGbellenot} are satisfied
 whenever  $\Lambda_0 < \Lambda
< 1$.

We now fix some $\Lambda$ such that $\Lambda_0<\Lambda<1$ and verify (3') to
(6').

 Affirmation (3') is obvious from Proposition \ref{bellenot} (3) for each $(\lambda_k,g)$, that is
$$m\norm{\cdot} \leqslant \norm{\cdot}_{\lambda_k,g} \leqslant \norm{\cdot},$$ 
and from Lemma \ref{LemmaGbellenot} (c), that is
$$\norm{\cdot}_{\Lambda,G}=\inf_{k \in K, g \in G}\norm{\cdot}_{\lambda_k,g}.$$

For (4') assume $1=\norm{y}>\norm{y}_{\Lambda,G}$. Then by Lemma
\ref{LemmaGbellenot} (d), there exist $g,k$ such that
$1=\norm{y}>\norm{y}_{\lambda_k,g}$, so from Proposition \ref{bellenot} (4) applied for $\norm{\cdot}_{\lambda_k,g}$,
$\norm{gx_k-y}<\delta_k$ or $\norm{-gx_k-y}<\delta_k$. This proves (4').

To prove (5') we note that if $\Norm{\frac x{\norm{x}}-x_k}<c_k/2$ then whenever
$gx_k \neq \pm hx_l$,
$$
\NORM{\frac x{\norm{x}}-gx_l} > \norm{gx_l-x_k}-c_k/2 \geqslant c_k/2  \geqslant \delta_k,
$$
and likewise 
$$
\NORM{\frac x{\norm{x}}+gx_l} \geqslant \delta_k.
$$
Applying Proposition \ref{bellenot} (4) to $y=x/\norm{x}$ and
$\norm{\cdot}_{\lambda_l,h}$ we deduce that 
$$
\norm{x}=\norm{x}_{\lambda_l,h}
$$
whenever $B_l^h \neq B_k^g$. From Lemma \ref{LemmaGbellenot} (c) 
it follows that 
$$
\norm{x}_{\Lambda,G}=\norm{x}_{\lambda_k,g}.
$$
Thus,  if  $\Norm{\frac x{\norm{x}}-x_k}<\epsilon_k$,  then $x$ is an isolated extremal point for
$\norm{\cdot}_{\Lambda,G}$ if and only if it is an isolated extremal   point for
$\norm{\cdot}_{\lambda_k,g}$, which, by an application of Proposition \ref{bellenot}
(5) for $x_k$ and $\epsilon_k$, is equivalent to saying that
$x=x_{k,\Lambda}$. Therefore, (5') is proved.

For the proof of (6'), denote by
$S_k^g$ the unit sphere for $\norm{\cdot}_{\lambda_k,g}$, by
$S_{\Lambda}^G$ the unit sphere for $\norm{\cdot}_{\Lambda,G}$,
by $S$ the unit sphere for $\norm{\cdot}$, $S'$ the set of points of $S$ on which
$\norm{\cdot}_{\Lambda,G}=\norm{\cdot}$. By Lemma \ref{LemmaGbellenot} (c) and (d),
$S_{\Lambda}^G=S' \cup (\bigcup_{k,g} (S_k^g \setminus S))$.
Let $[x_{k,\Lambda},x_{k,\Lambda}+w]$ be a maximal
line segment in $S_{\Lambda}^{G}$. 

We claim that $\|w\| \leqslant \delta_k$.
To see this assume that $\|w\|>\delta_k$, then we find a non trivial segment $[w_1,w_2]$ in
$S_{\Lambda}^G$ such that all points of the segment are at distance strictly more than $\delta_k$ and strictly less than $\delta_k+(d_k/12)$ to $x_{k,\Lambda}$, where
$d_k=\min_{i \leqslant k}c_i$.
Then, because of the lower estimate, no point on the segment may belong to $S_k^g \setminus S$, nor can it belong to $S_l^h \setminus S$ for $hx_l \neq gx_k$, as otherwise $d(S_k^g,S_l^h) < \delta_k+d_k/12<d_k/4+d_k/12=d_k/3$, contradicting Lemma \ref{LemmaGbellenot} (a). This means that the segment is included in $S'$. But this will then contradict the strict convexity of the norm ${\norm{\cdot}}$ in the hypothesis, since $\delta_k+(d_k/12) \leqslant d_k/3 \leqslant \epsilon_k/6$. 
So the claim is proved.

Going back to $x_{k,\Lambda}$, since this vector  belongs to $S_k^{\Id}\setminus S$, we 
deduce, from the claim and the fact that $\|y\|_{\Lambda,G}$ coincides with $\|y\|_{\lambda_k,Id}$ for any $y$ of $S_\Lambda^G$ with $\|x_{k,\Lambda}-y\| \leqslant 4\delta_k/3$, that 
 $[x_{k,\Lambda},x_{k,\Lambda}+w]$ should be a maximal line segment in
$S_k^{\Id}$.

By 
 Proposition \ref{bellenot}  (6) applied to
$\norm{\cdot}_{\lambda_k,\Id}$, we therefore deduce  that $b_k \geqslant \norm{w} \geqslant \lambda_k^{-1}-1$,
which proves (6') and concludes the proof.
 \end{proof}

%%%%%%%%%%%%%%%%%%%%%%%%%%%%%%%%%%%%%%%%%

\subsection{Distinguished sequences of vectors}

We now define a new notion.
For $G\leqslant {\rm Isom}(X)$ and $x_1,\ldots,x_n\in X$, denote by $G(x_1,\ldots,x_n)$ the closed subgroup
of $G$ which fixes all $x_i$, that is
$$G(x_1,\ldots,x_n)=\{g \in G\del  \forall i=1,\ldots n \;\; gx_i=x_i\},$$
and note that $G(\emptyset)=G$.

\begin{defi} 
Let $X$ be a  Banach space, $G$ be a group of isomorphims on $X$, and $(x_n)_n$ be a finite or infinite sequence of vectors in $X$. We shall say that $(x_n)_n$ is {\em distinguished} by $G$  if 
for any $n \geqslant 0$, the $G(x_0,\ldots,x_{n-1})$-orbit of $x_n$ is discrete. 
\end{defi}

A simple occurence of this is when 
the $G$-orbit of each $x_n$ is discrete. Note that when a sequence $(x_n)_n$ with dense linear span  is distinguished by $G$, we have that any two distinct elements $g$ and $g'$ of $G$ may be differentiated by their values $gx_n$ and $g'x_n$ in at least some $x_n$, and with some lower uniform estimate for $d(gx_n,g'x_n)$ depending only on $n$.

Observe also that the point $x_0$ must have discrete orbit for the sequence $(x_n)_n$ to be distinguished by $G$, but $x_0$ itself need not be a distinguished point, that is, some non-trivial elements of $G$ may fix $x_0$. 

\begin{lemme}\label{tec} Let $X$ be a Banach space, $G$ be a  group of isometries on $X$, and suppose that $(x_n)_{n \geqslant 0}$ is a possibly finite, linearly independent, $G$-distinguished sequence of points of $S_X$. Let $d_n$ be the distance of $x_n$ to $[x_0,\ldots,x_{n-1}]$ and let $\epsilon$ be positive. 
Then
there exist functions $\mu_0^n:\R^{n-1} \rightarrow \R$, $n \geqslant 1$, such that 
whenever a real positive sequence $(\mu_n)_n$ satisfies
$\mu_0=1$, $\mu_1 \leqslant \mu_1^0$ and $\mu_n \leqslant \mu_n^0(\mu_1,\ldots,\mu_{n-1})$ for all $n$, then the sequence $(y_n)_{n \geqslant 0}$ defined by
$$y_n=\frac{\sum_{0}^n \mu_k x_k}{\|\sum_{0}^n \mu_k x_k\|}$$
satisfies:
\begin{enumerate}
\item[(a)]  for any $n \geqslant 0$ and for any $g \in G$, either $y_n=gy_n$ or $\norm{y_n-gy_n} \geqslant (1-\epsilon)\mu_{n+1}d_{n+1}$,
\item[(b)]   for any $n > m \geqslant 0$ and for any $g \in G$,  $\|y_n-gy_m\| \geqslant (1-\epsilon) \mu_{m+1}d_{m+1}$.
\end{enumerate}
\end{lemme}

\begin{proof}
 Since the $G(x_0,\ldots,x_{n-1})$-orbit of $x_n$ is discrete for each $n$,
let $(\alpha_n)_n$ be a decreasing sequence of positive numbers such that for any $n$ and $g \in G(x_0,\ldots,x_{n-1})$, either
$gx_n=x_n$ or $\|gx_n-x_n\| \geqslant \alpha_n$. 

Observe also that by the definition of $(d_n)_n$ and Hahn-Banach theorem,
$$\norm{y+t x_{n}} \geqslant d_{n} |t|$$ whenever
$n \geqslant 1$, $y$ is in the linear span of $[x_0,\ldots,x_{n-1}]$ and $t \in \R$.

We may assume that $\epsilon<1$, and pick $\mu_0^k$, so that if 
$\mu_k \leqslant \mu_0^k(\mu_1,\ldots,\mu_{k-1})$ for each $k$, then
the sequence $(\mu_k)$ is decreasing sufficiently fast to ensure the following conditions, which may be redundant.
\begin{itemize}
\item for each $k$, $\|\sum_{j=0}^k  \mu_j x_j\| \in [1-\epsilon/2,1+\epsilon/2]$ and $\|\sum_{j=0}^k \mu_j x_j\|^{-1} \in [1-\epsilon/2,1+\epsilon/2]$,
\item for each $k$, $\sum_{j>k}\mu_j < \min\{\alpha_k \mu_k/4,\alpha_k/16,\mu_k\}$,
%\item for each $K$, $16\sum_{k>K}\mu_k < \alpha_K$,
%\item for each $K$, $16\sum_{k>K}\mu_k < 16\lambda_K$,
%\item for each $K$, $16 \sum_{k>K}\mu_k <2\epsilon \alpha_K$.
\item for each $k$,
$8 \mu_{k+1} \leqslant \min\{\alpha_k \mu_k/24,\epsilon d_k \mu_k\},$
\item for each $k$, 
$(1-\epsilon)d_{k+1}\mu_{k+1} \leqslant \alpha_k \mu_k/8.$ 
\end{itemize}

Now given $m \geqslant 0$,  we compute $\norm{y_n-gy_m}$, for $n \geqslant m$.

If $n=m$ then it is clear that either $gx_k=x_k$ for all $k \leqslant n$, in which case $y_n=gy_n$; or that, if $K=\min\{k \leqslant n: gx_k \neq x_k\}$, then by hypothesis  $\|gx_k-x_k\| \geqslant \alpha_K$, and therefore
$$\|y_n-gy_n\|=\frac{\|\sum_{k \geqslant K}\mu_k(x_k-gx_k)\|}{\|\sum_{k=0}^n \mu_k x_k\|} \geqslant \frac{1}{2}(\alpha_K \mu_K -2\sum_{k>K}\mu_k)$$
$$\geqslant \frac{1}{4} \alpha_K \mu_K \geqslant \frac{1}{4} \alpha_n \mu_n \geqslant
(1-\epsilon)\mu_{n+1}d_{n+1},$$
proving (a).

Assume now that $n>m$. Let $K$ be the first integer $k$  less than or equal to $m$ such that $gx_k \neq x_k$, if such an integer exists, or $K=m+1$ otherwise.
Then
$$
\|y_n-gy_m\|=\NORM{\frac{\sum_{k=0}^{K-1}\mu_k x_k+\sum_{k=K}^n \mu_k x_k}{\|\sum_0^n \mu_k x_k\|}-\frac{\sum_{k=0}^{K-1}\mu_k x_k+\sum_{k=K}^m \mu_k gx_k}{\|\sum_0^m \mu_k x_k\|}}. \ \ \ \ \ (1)
$$
First assume that $K \leqslant m$. Then the part of the sum in (1) corresponding to $x_K$ and $gx_K$ is equal to
$$
\frac{\mu_K x_K}{\|\sum_0^n \mu_k x_k\|}-
\frac{\mu_K gx_K}{\|\sum_0^m \mu_k x_k\|},
$$
whose norm is equal to 
$$
\NORM{\frac{\mu_K(x_K-gx_K)}{\|\sum_0^m \mu_k x_k\|}+
\Big(\frac{1}{\|\sum_0^n \mu_k x_k\|}-\frac{1}{\|\sum_0^m \mu_k x_k\|}\Big)\mu_K x_K},
$$
which is greater than
$$
\frac{\mu_K \alpha_K}2   -4\mu_K  \Big|\Norm{\sum_0^n \mu_k x_k}-\Norm{\sum_0^m \mu_k x_k}\Big| \geqslant
\frac{\mu_K}2\Big(\alpha_K-8\Norm{\sum_{m+1}^n \mu_k x_k}\Big) \geqslant
$$
$$
\frac{\mu_K}2 \Big(\alpha_K-8\sum_{k>m}\mu_k\Big) \geqslant
\frac{\mu_K}2 \Big(\alpha_K-8\sum_{k>K}\mu_k\Big)
\geqslant \mu_K \alpha_K/4.  \ \ \ \ \ (2)
$$

Now in the sum in (1) the linear combination of the $x_k$'s  for $k<K$ is equal to
$$
\Big(\frac{1}{\|\sum_0^n \mu_k x_k\|}-\frac{1}{\|\sum_0^m \mu_k x_k\|}\Big)
\sum_0^{K-1}\mu_k x_k,
$$
of norm less than or equal to
$$
8 
\Norm{\sum_{m+1}^n \mu_k x_k} \leqslant 8 \sum_{K+1}^m \mu_{k} \leqslant
16 \mu_{K+1}. \ \ \ \ \ (3) 
$$
The linear combination of $x_k$'s and $gx_k$'s for $k>K$ on the other hand, is equal to
$$
\frac{\sum_{K+1}^n \mu_k x_k}{\norm{\sum_0^n \mu_k x_k}}
- \frac{\sum_{K+1}^m \mu_k gx_k}{\norm{\sum_0^m \mu_k x_k}}
$$
of norm less than or equal to
$$
4\sum_{k>K}\mu_k \leqslant 8 \mu_{K+1}. \ \ \ \ \ (4)
$$
 Putting (2),(3) and (4) together we obtain that when $K \leq m$,
$$
\|y_n-gy_m\| \geqslant \mu_K \alpha_K/4 -24\mu_{K+1} \geqslant \mu_K \alpha_K/8
\geqslant \mu_m \alpha_m/8 \  \ \ \ \ (5)
$$

Assume now that $K=m+1$, that is, assume that $gx_k=x_k$ for all $k \leq m$, then $y_n -gy_m$ is the sum of two terms, namely,
$$y_n-gy_m=\Big(\frac{\sum_0^m \mu_k x_k+\mu_{m+1}x_{m+1}}{\|\sum_0^n \mu_k x_k\|}
-\frac{\sum_0^m \mu_k x_k}{\|\sum_0^m \mu_k x_k\|}\Big)+\Big(\frac{\sum_{k>m+1}\mu_k x_k}{\Norm{\sum_0^n \mu_k x_k}}\Big)$$

By definition of $d_{m+1}$, the  norm of the first term is at least
$$
\frac{d_{m+1}\mu_{m+1}}{\Norm{\sum_0^n \mu_k x_k}} \geqslant (1-\epsilon/2)d_{m+1}\mu_{m+1}, \ \ \ \ \ (6)
$$
while the norm of the second is at most
$$
2\sum_{k>m+1}\mu_k \leqslant 4 \mu_{m+2} \ \ \ \ \ (7).
$$
Putting (6) and (7) together we have that when $K=m+1$,
$$
\|y_n-gy_m\| \geqslant (1-\epsilon/2)d_{m+1}\mu_{m+1}-4 \mu_{m+2} \geqslant
(1-\epsilon)d_{m+1}\mu_{m+1}. \ \ \ \ \ (8)
$$

Now by (5) or (8) and the properties of the sequence $(\mu_n)_n$, we have that whenever $g \in G$ and $n>m$,
$\|y_n-gy_m\| \geqslant (1-\epsilon)d_{m+1}\mu_{m+1},$ which proves (b).
As $\epsilon$ was arbitrary this concludes the proof of the lemma. \end{proof}

\begin{thm}\label{main} 
Let $(X,\|\cdot\|)$ be a separable real Banach space and $G$ be a closed subgroup of ${\rm Isom}(X,\|\cdot\|)$ containing $-\Id$. Assume that there exists a possibly finite $G$-distinguished sequence $(x_n)_{n \geq 0}$ of points of $S_X$ with dense linear span, such that the norm $\|\cdot\|$ is LUR in a neighborhood of $x_0$.
Then there exists an equivalent norm $\triple\cdot$ on $X$ such that $G={\rm Isom}(X,\triple\cdot)$.
\end{thm}

\begin{proof}
By passing to a subsequence, we may without loss of generality assume that the sequence $(x_n)_n$ is linearly independent.

Choose a sequence $(\mu_n)_n$ satisfying the conditions defined in Lemma \ref{tec}. So, in particular, $\mu_0=1$ and $4\sum_{k>n}\mu_k<\mu_n \alpha_n$, where  $\alpha_n>0$ is fixed such that any $g \in G_n$ satisfies $gx_n=x_n$ or $\|gx_n-x_n\| \geqslant \alpha_n$.
Set $y_0=x_0$ and for $n \geqslant 1$ let  $z_n=\sum_{k=0}^n \mu_k x_k$ and  $y_n=\frac{z_n}{\|z_n\|}$.
By choosing the $\mu_n$ sufficiently small, we may therefore also assume that every $y_n$ belongs to an open neighborhood of $x_0$ where the norm is LUR.

By the fact that the norm is LUR in a neighborhood of every $y_n$, condition (1) for the sequence $(y_n)_n$ in Proposition \ref{Gbellenot} holds, and conditions (2)-(3)  follow from
the conclusion of Lemma \ref{tec}.  So let $(\lambda_n)$ be a sequence satisfying the conditions given by the conclusion of Proposition \ref{Gbellenot} for the sequence $(y_n)_n$, let $\Lambda=(\lambda_n)_n$ and consider the associated $\Lambda,G$-pimple norm $\triple \cdot$ . As was already observed, any $g$ in $G$ is still an isometry for the norm $\triple\cdot$.

Let now $T$ be a $\triple\cdot$-isometry. We shall prove that $T$ belongs to $G$.
Observe that $E=\{\lambda_n^{-1}gy_n\del  g \in G, n \in K\}$ is the set of isolated extremal points of $\triple\cdot$. Indeed
for an  $x$ of $\Lambda,G$-pimple norm $1$, either $\Norm{\frac{x}{\norm{x}}-gy_n}<c_n/2$ for some $g,n$, in which case by (5') $x=\lambda_n^{-1}y_n$ if it is an isolated extremal point; or
$\Norm{\frac{x}{\norm{x}}-gy_n} \geqslant c_k/2 >\delta_n$ for all $g,n$ then by
(4') the $\Lambda,G$-pimple norm coincides with ${\norm{\cdot}}$ in a neighborhood of $x$ and then $x$ is not an isolated extremal point since $\norm{\cdot}$ is strictly convex in $x$.

Therefore any isometry $T$ for $\triple\cdot$ maps $E$ onto itself. If $n<m$ and $g \in G$, then $T$ cannot map $\lambda_n^{-1}y_n$ to
$\lambda_m^{-1} gy_m$. Indeed if $w$ (resp. $w'$) is a vector so that
$\lambda_n^{-1}y_n$ and $\lambda_n^{-1}y_n+w$ (resp.
$\lambda_m^{-1}gy_m$ and $\lambda_m^{-1}gy_m+w'$) are endpoints of a maximal
line segment in $S_X^{\triple\cdot}$, 
then, since $g$ is a $\triple\cdot$-isometry, we may assume $g= \Id$ and  by (3') and (6')
$$
\triple w \geqslant \frac{1}{2} \norm{w} \geqslant\frac{1}{2}(\lambda_n^{-1}-1)>b_{n+1} \geqslant b_m \geqslant \norm{w'} \geqslant \triple w',
$$
and $\triple w \neq \triple w^{\prime}$, a contradiction.
Likewise, if $n>m$, we may by using $T^{-1}$ deduce that $T$ cannot map 
$\lambda_n^{-1}y_n$ to $\lambda_m^{-1}gy_m$.

Finally it follows that for each $n$, the orbit $Gy_n$ is preserved by $T$.
Then there exists for each $n$ a $g_n$ in $G$ such that $Ty_n=g_n y_n$, or equivalently
$$\sum_{k=0}^n \mu_k(Tx_k-g_nx_k)=0.$$
Now given $n$ we claim that $Tx_k=g_nx_k$ for every $k \leqslant n$.
This is clear for $n=0$. Assuming this holds for $n-1$, the previous equality becomes
$$\sum_{k=0}^{n-1}\mu_k(g_{n-1}x_k-g_nx_k)+(Tx_n-g_nx_n)=0.$$
Now if $Tx_k=g_{n-1}x_k=g_nx_k$ for all $k \leqslant n-1$, then it follows that $Tx_n-g_nx_n=0$ and we are done. Assume therefore that
there exists $k \leqslant n-1$ such that $g_{n-1}x_k \neq g_n x_k$ and let $K$ be the first such $k$, we would deduce that
$$
\mu_K(g_{n-1}x_K-g_nx_K)=-\sum_{k>K}\mu_k(Tx_k-g_nx_k).
$$
Since $g_n^{-1}g_{n-1}$ belongs to $G(x_0,\ldots,x_{K-1})$ and $g_n^{-1}g_{n-1}x_K \neq x_K$, the left hand part has ${\norm{\cdot}}$-norm at least $\mu_K  \alpha_K$.
 The right hand part has ${\norm{\cdot}}$-norm at most $3\sum_{k>K}\mu_k$.
By the condition on $(\mu_n)_n$ this is a contradiction. So for any $n$ and any $k \leqslant n$, $Tx_k=g_n x_k$. 

If the distinguished sequence $(x_n)_n$ is finite, of the form $(x_n)_{0 \leqslant n \leqslant N}$, then we deduce that $T=g_N \in G$. If it is infinite, then for each $k \in \N$ we deduce that $\lim_n g_n x_k=Tx_k$. Since $X$ is the closed linear span of the $x_k$' s this means that $g_n$ converges to $T$ in the strong operator topology. 
Since $G$ is closed we conclude that $T$ must belong to $G$.
\end{proof}

%%%%%%%%%%%%%%%%%%%%%%%%%%%%%%%%%%%%

%%%%%%%%%%%%%%%%%%%%%%%%%%%%%%%%%%%%%%%%%%%%%%%%%%%%%%%%%%% 

\subsection{Displays on the  spaces $\ell_p$, $L_p$ and $C([0,1])$}
 Note that if we let $S_\infty$ act on the compact group $\Z_2^\N$ by permutations of the coordinates, we can form the corresponding topological semidirect product $S_\infty\ltimes \Z_2^\N$. In other words, $S_\infty\ltimes \Z_2^\N$ is the cartesian product $S_\infty\times \Z_2^\N$ equipped with the product topology and the following group operation 
$$
(g,(a_n)_{n\in \N})*(f,(b_n)_{n\in \N})=(gf, (a_nb_{g\inv(n)})_{n\in \N}),
$$
where $g,f\in S_\infty$ and $(a_n)_{n\in \N},(b_n)_{n\in \N}\in \Z_2^\N$. In the following, we shall identify $\Z_2$ with the multiplicative group $\{-1,1\}$.

Suppose $X$ is a Banach space with a {\em symmetric Schauder decomposition} $X=\sum_{n\in \N}X_n$, i.e., such that for some equivalent norm $\norm\cdot$ the $X_n$ are isometric copies of the same space $Y\neq \{0\}$ and there is a symmetric Schauder basis $(e_n)_{n\in \N}$ for which
$$
\Norm{\sum_{n\in \N}x_n}=\Norm{\sum_{n\in \N}\norm{x_n}e_n},
$$
for $x_n\in X_n$. Then we can define a topologically faithful bounded linear representation $\rho\colon S_\infty\ltimes \Z_2^\N\til GL(X)$ by
setting
$$
\rho\big(g, (a_n)_{n\in \N}\big)\Big(\sum_{n\in \N}x_n\Big)=\sum_{n\in \N}a_nx_{g\inv(n)}.
$$
In other words, $\rho$ is the canonical action of $S_\infty\ltimes \Z_2^\N$ on $X=\sum_{n\in \N}X_n$ by change of signs and permutation of coordinates. Since the representation is bounded and topologically faithful it is easy to check that $\rho\big(S_\infty\ltimes \Z_2^\N\big)$ is a \textsc{sot}-closed subgroup  of $GL(X)$, though it is not \textsc{sot}-closed in $L(X)$.

\begin{lemme}\label{sinfty}
Let $G$ be a closed subgroup of $S_\infty$ containing a non-trivial central involution $s$. Then there is a topological embedding $\pi\colon G\til S_\infty \ltimes \Z_2^\N$ such that $\pi(s)=(1_{S_\infty}, (-1,-1,\ldots))$.
\end{lemme}

\begin{proof}
As $s$ is central, the set
$$
B=\{n\in \N\del s(n)\neq n\}
$$
is $G$-invariant, so we can define a faithful action of $G$ on $B \times \N$ by $g(n,m)=(g(n),g(m))$. It follows that $G$ is topologically isomorphic to its image in the group ${\rm Sym}(B\times \N)$ of all permutations of the countable set $B\times \N$ and, moreover, $s(n,m)\neq (n,m)$ for any $(n,m)\in B\times \N$. 

Let $(O_n)_{n \in \N}$ denote the orbits of $s$ on $B\times \N$ and note that, as $s$ is an involution without fixed points, $|O_n|=2$ for all $n\in \N$. Moreover, since $s$ is central, $G$ permutes the orbits of $s$ and so we obtain a homomorphism $\sigma\colon G\til S_\infty$ given by
$$
\sigma_g(n)=k \;\;\;\equi \;\;\;g[O_n]=O_k.
$$
Also, if  $\prec$ denotes a linear ordering on $B\times \N$, we can define a map $\rho\colon G\til \Z_2^\N$ by 
$$
\rho(g)_n=\begin{cases}1 & \text{ if $g\colon O_{\sigma_g\inv(n)}\mapsto O_n$ preserves $\prec$, }\\
-1 & \text{ if $g\colon O_{\sigma_g\inv(n)}\mapsto O_n$ reverses $\prec$.}
\end{cases}
$$
It follows that  the map $\pi\colon G\til S_\infty \ltimes \Z_2^\N$ defined by
$$
\pi(g)=(\sigma_g, (\rho(g)_n)_{n\in \N})
$$
is a embedding of $G$ into $S_\infty \ltimes \Z_2^\N$. To see this, note that, as the action of $G$ on $B\times \N$ is faithful, the map is clearly injective. On the other hand, to verify that it is a homomorphism, note that for $g, f\in G$, 
\[\begin{split}
\pi(g)*\pi(f)&=\big(\sigma_g,\big (\rho(g)_n\big)_{n\in \N}\big)*\big(\sigma_f,\big (\rho(f)_n\big)_{n\in \N}\big)\\
&=\big(\sigma_g\sigma_f,\big (\rho(g)_n\rho(f)_{\sigma_g\inv(n)}\big)_{n\in \N}\big)\\
&=\big(\sigma_{gf},\big (\rho(gf)_n\big)_{n\in \N}\big)\\
&=\pi(gf).
\end{split}\]
Finally, $\pi$ is now easily checked to be a homeomorphism of $G$ with its image in $S_\infty\ltimes \Z_2^\N$ and hence a topological  group embedding. Moreover, $\pi(s)=\big(1_{S_\infty}, (-1,-1,\ldots)\big)$. 
\end{proof}

\begin{lemme}\label{sym} 
Let $X$ be a separable real Banach space with an LUR norm, which is a Schauder sum of an infinite number of isometric copies of some space $Y$, and let $G$ be represented as a closed subgroup of ${\rm Isom}(X)$ containing $-\Id$ and acting by change of signs and/or permutations of the coordinates relative to the Schauder sum. Then the representation is a display. 
\end{lemme}

\begin{proof} We may pick a sequence $(x_n)$ with dense linear span and such that each $x_n$ is in some summand of the decomposition. It follows that $Gx_n$ is discrete for each $n$. Therefore $(x_n)$ is a distinguished sequence with dense linear span and hence Theorem \ref{main} applies.
\end{proof}

\begin{thm} Any closed subgroup $G\leqslant S_{\infty}$ with a non-trivial central involution is displayable on $c_0$,  $C([0,1])$, $\ell_p$ and $L_p$,  for $1 \leqslant p<\infty$.
\end{thm}

\begin{proof} 
The spaces $c_0$ and $\ell_p$ have a symmetric basis, while $L_p$ and $C([0,1])$ have symmetric decompositions of the form $C([0,1]) \simeq c_0\big(C([0,1])\big)$ and $L_p=\ell_p(L_p)$. So let $X$ be one of these and let $X=\sum_{i\in \N}X_i$ be the corresponding Schauder decomposition into isometric copies $X_i$ of some space $Y$. 
We note that in \cite{FG} these spaces are shown to admit equivalent LUR norms such that $\rho \colon S_\infty\ltimes \Z_2^\N\til GL(X)$ defined above is an isometric representation (the original norm will do for $\ell_p$, $1<p<\infty$, and Day's norm will do for $c_0$). 
Moreover, by Lemma \ref{sinfty}, $G$ can be seen as a closed subgroup of $\rho(S_\infty\ltimes \Z_2^\N)$ containing $-\Id$ and hence by Lemma \ref{sym} this identification is a display of $G$ on $X$.
\end{proof}

Observe that if  $X$ is a separable Banach space and $G$ is a group of isometries on $X$ which  distinguishes some point $y\in X$, then it will distinguish any sequence of $X$ all of whose terms are sufficiently close to $y$.
This is an easy consequence of the fact that the set of points (not sequences) distinguished by $G$ is open in $X$. So $G$ will distinguish some sequence with dense linear span. Therefore,  we have obtained an easier proof of a result of \cite{FG}.

\begin{cor} Let $X$ be a separable real Banach space and $G$ be a group of isometries on $X$ containing $-\Id$. Assume that $X$ contains a  point $y$ distinguished by $G$, and that the norm is LUR in some neighbourhood of $y$. Then there exists an equivalent norm on $X$ such that $G={\rm Isom}(X,{\norm{\cdot}})$.
\end{cor}

Suppose now that $G$ is a bounded group of automorphisms of $X$ containing $-\Id$. Then $G$ will be a group of isometries for the equivalent norm $\sup_{g \in G}\|gx\|$. Furthermore, a result of G. Lancien in  \cite{L} asserts that any separable Banach space with the Radon-Nikodym Property (RNP) may be renormed with an equivalent LUR norm without diminishing the group of isometries. Using this fact we obtain generalizations of the previous results to the case of bounded groups of isomorphisms on spaces with the RNP. For example, Theorem \ref{main} implies the following.

\begin{cor} Let $X$ be a separable real Banach space with the Radon-Nikodym Property and $G$ be a closed bounded subgroup of $GL(X)$ containing $-\Id$. Assume that some sequence in $X$ is distinguished by $G$ and has dense linear span. Then there exists an equivalent norm
$\norm{\cdot}$ on $X$ such that $G={\rm Isom}(X,{\norm{\cdot}})$.
\end{cor}

%%%%%%%%%%%%%%%%%%%%%%%%%%%%%%%%%%%%%%%%%%%%%%%%

\subsection{Displays and representations}

We shall now see that under certain mild conditions on a separable space $X$, any bounded, countable closed (or equivalently discrete), subgroup of $GL(X)$  containing $-\Id$ is  displayable on some power  of $X$.

\begin{prop}\label{Xn} Let $X$ be a separable real Banach space with an LUR norm and  $G$ be a discrete group of isometries on $X$ containing $-\Id$. Then $G$ is displayable on $X^n$ for some $n \in \N$.
\end{prop}

\begin{proof}
Since $\Id$ is isolated in $G$, 
we may find $\alpha>0$ and $x_1,\ldots,x_n \in X$ such that for 
any $g \neq \Id$, $\norm{gx_i-x_i} \geqslant \alpha$ for some $i=1,\ldots,n$.
If we let $G$ act on $X^n$ by $g(y_1,\ldots,y_n)=(gy_1,\ldots,gy_n)$ and equip $X^n$ with the $\ell_2$-sum of the norm on $X$, we see that $G$ can be identified with a closed subgroup of ${\rm Isom}(X^n)$ containing $-\Id$ and such that the point $x_0=(x_1,\ldots,x_n)$ is distinguished by $G$. Since, by classical arguments, the norm on $X^n$ is again LUR (see Fact 2.3 p. of \cite{DGZ}), Theorem \ref{FG} then applies to prove that $G$ is displayable on $X^n$.
\end{proof}

\begin{cor} Let $X$ be a separable real Banach space with an LUR norm and isomorphic to its square. Then any discrete group of isometries on $X$ containing $-\Id$ is displayable on $X$.
\end{cor}

Note here that the original representation itself need not be a display.

\begin{prop} Let $X$ be a separable real Banach space with the Radon-Nikodym Property and $G$ be a discrete bounded subgroup of $GL(X)$ containing $-\Id$. Then $G$ is displayable on $X^n$ for some $n \in \N$.
\end{prop}

\begin{proof} We renorm $X$ with the norm $\sup_{g \in G}\|gx\|$ to make any element of $G$ into an isometry. Since $X$ has the Radon-Nikodym Property, we may apply the result of Lancien to obtain a LUR norm on $X$ for which each element of $G$ is still an isometry. Then we may apply Proposition \ref{Xn}.
\end{proof}

Finally we may conclude that for spaces with the Radon-Nikodym Property and isomorphic to their square, and for countable groups, the necessary conditions of the second section are sufficient for the existence of a display.

\begin{cor} Let $X$ be separable real Banach space with the Radon-Nikodym Property and isomorphic to its square, and let $G$ be a countable topological group. Then $G$ is displayable on $X$ if and only if it is faithfully topologically representable as a discrete bounded subgroup of $GL(X)$ containing $-\Id$.
\end{cor}

%%%%%%%%%%%%%%%%%%%%%%%%%%%%%%%%%%%%%%%%%%%%%

\section{Continuous groups of transformations of $\R^2$}

In this section we look at the displayability some of the simplest uncountable compact metric groups, namely, the circle group $\T=\{z\in \C\del |z|=1\}$, and the orthogonal group $O(2)=\{U\in M_2(\R)\del UU^t=U^tU=1\}$.

Recall that a real Banach space $X$ is said to admit a {\em complex structure} when there exists some operator $J$ on $X$ such that $J^2=-\Id$. This means that $X$ may be seen as a complex space, with the scalar multiplication  defined by $$(a+ib)\cdot x=ax+bJx$$ and the equivalent renorming
$$\triple{x}=\sup_{0 \leqslant \theta \leqslant 2\pi}\|\cos \theta x +\sin \theta Jx\|.$$

If $X$ is isomorphic to a Cartesian square $Y^2$, then it admits a complex structure,
associated to the operator $J$ defined by $J(y,z)=(-z,y)$ for  $y,z \in Y$.  But examples of real spaces admitting a complex structure without being isomorphic to a Cartesian square also exist, see e.g. \cite{F} about these.

\begin{thm} Let $X$ be a real Banach space. Then $\T$  is displayable on $X$ if and only if $X$ admits a complex structure and $\dim X>2$.
\end{thm}

\begin{proof} The `if' part was proved in \cite{FG}, Corollary 45. 
For the `only if' part, assume $\rho\colon \T\til GL(X)$ is a display of $X$ and let $\norm\cdot$ be the corresponding norm on $X$. Then, as $-1$ is the unique non-trivial involution in $\T$, we must have $\rho(-1)=-\Id$, whereby $\rho(i)^2=\rho(i^2)=-\Id$ and hence $X$ will admit a complex structure.

Now, assume towards a contradiction that $\dim X=2$, i.e., $X=\R^2$. Since $\T$ is compact, by  a standard result of representation theory, $\rho$ is orthogonalisable, meaning that there is a $T\in GL(\R^2)$ such that $T\rho(\lambda)T\inv$ is an isometry for the euclidean norm $\norm\cdot_2$ for any $\lambda\in \T$. In other words, 
$$
T\rho T\inv \colon \T\til O(2)
$$
and, as $\T$ is connected, so is its image, which thus must lie in the connected component of $\Id$, which is the group $SO(2)\isom\T$ of rotations of $\R^2$. Again, as $T\rho T\inv$ has non-trivial image, it must be surjective.

On the other hand, if $U\in O(2)\setminus SO(2)$, then for any $x\in \R^2$ there is $B\in SO(2)$ such that $U(Tx)=B(Tx)$, and, therefore, if $T\rho(\lambda) T\inv=B$, 
$$
\norm{T\inv UTx}=\norm{T\inv BTx}=\norm{\rho(\lambda)x}=\norm x.
$$
It follows that $T\inv UT$ is an isometry for $\norm\cdot$ not belonging to $\rho(\T)$, contradicting that $\rho(\T)={\rm Isom}(\R^2,\norm\cdot)$.
\end{proof}

\begin{thm} Let $X$ be a non-trivial real Banach space. Then $O(2)$ is displayable on $X$ if and only if $X$ is isomorphic to a Cartesian square.
\end{thm} 

\begin{proof} The `if' part was proved in \cite{FG}, Corollary 46. 

For the `only if' part, assume that $\rho \colon O(2)\til GL(X)$ is a display of the group $O(2)={\rm Isom}(\R^2,\norm\cdot_2)$ on $X$.  Since $-\Id_{\R^2}$ is the unique non-trivial central involution in $O(2)$, we have $\rho(-\Id_{\R^2})=-\Id$. So, if $R$ denotes the rotation of angle $\pi/2$ on $\R^2$ and   $J=\rho(R)$, we have $J^2=-\Id$. Also, letting $S$ denote the reflection of $\R^2$ about the $x$-axis and $M=\rho(S)$, we have $JMJ=\rho(RSR)=M$.

Now, as $M^2=\Id$, $P=(\Id+M)/2$ is idempotent, i.e., a projection with complementary projection $Q=(\Id-M)/2$ and associated decomposition $X=Y \oplus Z$, $Y=PX$ and $Z=QX$. 

Define $T \in L(Y,Z)$ by
$Ty=QJy$ and $U \in L(Z,Y)$ by $Uz=-PJz$.
Then for any $y \in Y$,
$$
2UTy=-2PJQJy=-PJ(\Id-M)Jy=P(\Id+JMJ)y=P(\Id+M)y=2Py=2y,
$$
whence $UT= \Id_Y$. Likewise $TU= \Id_Z$, which implies that $Y$ and $Z$ are isomorphic.
In other words, $X$ is isomorphic to a Cartesian square.
\end{proof}

%%%%%%%%%%%%%%%%%%%%%%%%%%%%%%%%%%%%%%%%%%%%%%%%%%%

\section{Representation of Polish groups and universality}

\begin{thm} \label{universal}
For any Polish group $G$, there exists a separable Banach space $X$ such that $\{-1,1\} \times G$ is displayable on $X$.
\end{thm}

\begin{proof}
Let $G$ be a Polish group. Then, by Theorem 3.1(i) of \cite{GK}, there is a separable complete metric space $(Y,d_1)$ such that 
$$
G\isom {\rm Isom}(Y,d_1),
$$
where the latter is equipped with the topology of pointwise convergence on $Y$.  Moreover, without loss of generality, we can suppose that $Y$ contains at least $3$ points.

Let now $d_2=\frac {d_1}{1+d_1}$. Then $d_2$ is a complete metric on $Y$ inducing the same topology and, moreover, since the function $t\mapsto \frac t{1+t}$ is strictly increasing from $[0,\infty[$ to $[0,1[$, 
$$
{\rm Isom}(Y,d_2)={\rm Isom}(Y,d_1).
$$

Again, let $d_3=\sqrt d_2$ and note that, by the same arguments as for $d_2$, $d_3$ is a complete metric on $Y$ with $d_3<1$ and 
$$
{\rm Isom}(Y,d_3)={\rm Isom}(Y,d_2).
$$
Furthermore, in the terminology of \cite{weaver}, $d_3$ is a {\em concave} metric on $Y$. 

Let now  $\AE(Y,d_3)$ denote the $\R$-vector space of all finitely supported functions $m\colon Y\til \R$ satisfying
$$
\sum_{y\in Y}m(y)=0,
$$
so formally $\AE(Y,d_3)$ can be identified with a hyperplane in the free $\R$-vector space over the basis $Y$. Elements of $\AE(Y,d_3)$ are called {\em molecules} and basic among these are the {\em atoms}, i.e., the molecules of the form
$$
m_{x,y}=\1_x-\1_y,
$$
where $x,y\in Y$ and $\1_x$ is $1$ on $x$ and $0$ elsewhere and similarly for $\1_y$. As can easily be seen by induction on the cardinality of its support, any molecule $m$ can be written as
$$
m=\sum_{i=1}^na_im_{x_i,y_i},
$$
for some $x_i,y_i\in Y$ and $a_i\in \R$.

We equip $\AE(Y,d_3)$ with the norm ${\norm\cdot}_{\AE}$, defined by
$$
{\norm {m}}_{\AE}=\min\big( \sum_{i=1}^n |a_i| d_3(x_i,y_i) \del m=\sum_{i=1}^na_im_{x_i,y_i}\big),
$$
and remark that the norm is equivalently computed by
$$
\norm m_{\AE}=\sup\big( \sum_{y\in Y}m(y)f(y)\del f\colon (Y,d_3)\til \R \text{ is $1$-Lipschitz }\big).
$$
Abusing notation, the completion of $\AE(Y,d_3)$ with respect to $\norm\cdot_{\AE}$ will also be denoted by $\AE(Y,d_3)$. 
It is not difficult to verify that the set of molecules that are rational linear combinations of atoms with support in a countable dense subset of $Y$ is dense in $\AE(Y,d_3)$ and thus  $\AE(Y,d_3)$ is a separable Banach space.

The space $\AE(Y,d_3)$ is called the {\em Arens--Eells} space of $Y$ and a fuller account of its properties and uses can be found in N. Weaver's book \cite{weaver}. Of particular importance to us is the following result relying on the concavity of $(Y,d_3)$, here stated slightly differently from Theorem 2.7.2 in \cite{weaver}. Namely, for any surjective linear isometry 
$$
T\colon \AE(Y,d_3)\til \AE(Y,d_3)
$$
there are some $\sigma=\pm 1$, $\lambda>0$ and a bijective $\lambda$-dilation $g\colon (Y,d_3)\til (Y,d_3)$, i.e., $d_3(gx,gy)=\lambda d_3(x,y)$ for all $x,y\in Y$, such that for any molecule $\sum_{i=1}^na_im_{x_i,y_i}$,
$$
T\big(\sum_{i=1}^na_im_{x_i,y_i}\big)= \frac\sigma\lambda\sum_{i=1}^na_im_{gx_i,gy_i}.
$$
However, as $(Y,d_3)$ has finite diameter $>0$, there are no surjective $\lambda$-dilations from $Y$ to $Y$ for $\lambda\neq 1$, from which it follows that  $g\in {\rm Isom}(Y,d_3)$. 

We claim that since $|Y|\geqslant 3$, $\sigma=\pm 1$ and $g\in {\rm Isom}(Y,d_3)$ are uniquely determined as a function of $T$. To see this, note that since
$$
T(m_{x,y})=\sigma m_{gx,gy},
$$
we have $\{gx,gy\}=\{x,y\}$ for all $x,y\in Y$. So the action of $g$ on doubletons is uniquely given, whence, as $|T|\geqslant 3$, the action of $g$ on $Y$ is also completely determined by $T$. Moreover, once $g$ is given,  $T(m_{x,y})=\sigma m_{gx,gy}$ also describes $\sigma$ in terms of $T$.

Thus, $T\mapsto (\sigma,g)$ is a function from ${\rm Isom}(\AE(Y,d_3))$ to $\{-1,1\}\times {\rm Isom}(Y,d_3)$, which is easily seen to be injective, and considering 
$$
T(m_{x,y})=\sigma m_{gx,gy},
$$
one finds that it is a group homomorphism. Conversely, as any $(\sigma,g)\in \{-1,1\}\times {\rm Isom}(Y,d_3)$ defines a linear isometry $T$ of $\AE(Y,d_3)$ by the above, the map $T\mapsto (\sigma,g)$ is an isomorphism of ${\rm Isom}(\AE(Y,d_3))$ with $\{-1,1\}\times {\rm Isom}(Y,d_3)$.
Being a definable isomorphism between Polish groups, it must be a topological isomorphism and, similarly, ${\rm Isom}(Y,d_3)$ is topologically isomorphic with $G$. It follows that $\{-1,1\}\times G$ can be displayed as ${\rm Isom}(\AE(Y,d_3))$.
\end{proof}

The above construction coupled with a well-known result of V.V. Uspenskij \cite{uspenskii} immediately gives us the existence of a separable Banach space whose isometry group is universal for all Polish groups. Namely, let $\U$ denote the so-called {\em Urysohn} metric space, which is the unique (up to isometry) separable complete metric space such that any isometry between finite subsets extends to a full isometry of $\U$. Then, as shown in \cite{uspenskii}, any Polish group $G$ embeds as a closed subgroup into ${\rm Isom}(\U)$.
 
\begin{thm} 
Any Polish group $G$ is topologically isomorphic to a closed subgroup of ${\rm Isom}(\AE(\U))$, where $\U$ is the Urysohn space.
\end{thm}

\begin{proof}
It suffices to show that ${\rm Isom}(\U)$ embeds into ${\rm Isom}(\AE(\U))$, since, being Polish, the image of ${\rm Isom}(\U)$ will automatically be closed.

So let $T\colon {\rm Isom}(\U)\til {\rm Isom}(\AE(\U))$ be defined by 
$$
T_g\big(\sum_{i=1}^na_im_{x_i,y_i}\big)= \sum_{i=1}^na_im_{gx_i,gy_i}
$$
for any molecule $\sum_{i=1}^na_im_{x_i,y_i}$, $g\in {\rm Isom}(\U)$, extending $T_g$ continuously to the completion $\AE(\U)$.
\end{proof}

%%%%%%%%%%%%%%%%%%%%%%%%%%%%%%%%%%%%%%%%%%%%%%%%%%%%

\section{Transitivity of norms and LUR renormings}

In this section, we relate the question of diplayability of groups on Banach spaces  to the classical theory of transitivity of norms. A norm on a space $X$ is said to be {\em transitive} if for any $x,y$ in the unit sphere of $X$, there exists a surjective isometry $T$ on $X$ such that $T(x)=y$. Whether any separable Banach space with a transitive norm must be isometric to $\ell_2$ is a longstanding open problem known as the  {\em Mazur rotation problem}.

We begin by recalling weaker notions of transitivity introduced by A. Pe\l czy\'nski and S. Rolewicz at the International Mathematical Congress in Stockholm in 1964 (see also \cite{R}). After this, we shall clarify the relations between LUR renormings and  transitivity of norms. Related results may be found in some earlier work of F. Cabello-Sanchez \cite{CS} and of J. Becerra Guerrero and A. Rodr\'iguez-Palacios \cite{BR}.

\begin{defi} Let $(X,{\norm{\cdot}})$ be a Banach space and for any $x \in  S_X$, let ${\mathcal O}(x)$ denote the orbit of $x$ under the action of ${\rm Isom}(X,{\norm{\cdot}})$. The norm ${\norm{\cdot}}$ on $X$ is
\begin{itemize}
\item[(i)] {\em transitive} if for any $x \in S_X$, ${\mathcal O}(x)=S_X$,
\item[(ii)] {\em almost transitive} if for any $x \in S_X$, ${\mathcal O}(x)$ is dense in $S_X$,
\item[(iii)] {\em convex transitive} if for any $x \in  S_X$, ${\rm conv}{\mathcal O}(x)$ is dense in $B_X$,
\item[(iv)] {\em uniquely maximal} if whenever ${\triple{\cdot}}$ is an equivalent norm on $X$ such that
${\rm Isom}(X,{\norm{\cdot}}) \subseteq {\rm Isom}(X,{\triple{\cdot}})$, then ${\triple{\cdot}}$ is a multiple of ${\norm{\cdot}}$.
\item[(v)] {\em maximal} if whenever ${\triple{\cdot}}$ is an equivalent norm on $X$ such that
${\rm Isom}(X,{\norm{\cdot}}) \subseteq {\rm Isom}(X,{\triple{\cdot}})$, then ${\rm Isom}(X,{\norm{\cdot}})={\rm Isom}(X,{\triple{\cdot}})$.
\end{itemize}
\end{defi}

The implications transitive $\Rightarrow$ almost transitive $\Rightarrow$ convex transitive, as well as uniquely maximal $\Rightarrow$ maximal  are obvious, and  convex transitivity and unique maximality are equivalent by a result of E. Cowie \cite{C}.

Given a Banach space $X$ with norm ${\norm{\cdot}}$, we shall say that an equivalent norm ${\triple{\cdot}}$ on $X$ {\em does not diminish the group of isometries} when 
$${\rm Isom}(X,{\norm{\cdot}}) \subseteq {\rm Isom}(X,{\triple{\cdot}}).
$$

G. Lancien \cite{L} has proved that if a separable Banach space $X$ has the Radon-Nikodym Property, then there exists an equivalent LUR norm on $X$, which does not diminish the group of isometries; and that if $X^*$ is separable, then there exists an equivalent norm on $X$ whose dual norm is LUR and which does not diminish the group of isometries. In \cite{FR} it is observed that if $X$ satisfies the two properties, then one can find a norm not diminishing the group of isometries, which is both LUR and with LUR dual norm. 

Note that it is a classical result of renorming theory, due to M. Kadec, see \cite{DGZ}, that any separable space admits an equivalent LUR norm, but of course, this renorming may alter the group of isometries, and we shall actually see that the results of Lancien do not extend to all separable spaces.

\begin{lemme}\label{atn} 
If an almost transitive norm on a Banach space is LUR in some point of the unit sphere, then it is uniformly convex. 
\end{lemme}

\begin{proof} Fix $x_0 \in S_X$ in which the norm is LUR. For any $\epsilon>0$, let $\delta>0$ such that
$$\|x-x_0\| \geqslant \epsilon/2 \Rightarrow \|x+x_0\| \leqslant 2-\delta.$$
Let $x$ be arbitrary in $S_X$, and let $g \in G$ be such that $\|x-gx_0\| \leqslant \max(\epsilon/2,\delta/2)$.
For any $y$ in $S_X$ such that $\|x-y\| \geqslant \epsilon$, we deduce $\|y-gx_0\| \geqslant \epsilon/2$ and therefore, since $g$ is an isometry,
$$\|y+gx_0\| \leqslant 2-\delta,$$
whereby
$$\|y+x\| \leqslant \|y+gx_0\|+\|gx_0-x\| \leqslant 2-\delta/2.$$
This proves that the norm is uniformly convex. \end{proof}

\begin{lemme}\label{luruc} 
If a convex transitive norm on a Banach space is LUR on a dense subset of the unit sphere, then it is almost transitive and uniformly convex.
\end{lemme}

\begin{proof} 
By Lemma \ref{atn}, it suffices to prove that the norm is almost transitive. Let $X$ be a Banach space and set $G={\rm Isom}(X)$. Assume that the norm is LUR on a dense subset of $S_X$ and not almost transitive, and let $x_0,x \in S_X$ and $\epsilon>0$ be such that the orbit $Gx_0$ is at distance at least $\epsilon$ from $x$. By density we may assume that the norm is LUR in $x$, so let $\alpha>0$ be such that $\|x+y\| \leqslant 2-\alpha$ whenever $y \in S_X$ is such that $\|x-y\| \geqslant \epsilon$. Let $\phi$ be a normalized functional such that $\phi(x)=1$. Then whenever $g \in G$,
$$\|x-gx_0\| \geqslant \epsilon,$$
and therefore
$$\phi(x+gx_0) \leqslant \|x+gx_0\| \leqslant 2-\alpha,$$
whereby 
$$\phi(gx_0) \leqslant 1-\alpha.$$
It follows that for any $y$ in the closed convex hull of $Gx_0$, $\phi(y) \leqslant 1-\alpha$, which proves that $x$ does not belong to the closed convex hull of $Gx_0$. Therefore the norm is not convex transitive.
\end{proof}

\begin{quest} 
Let $X$ be a separable real Banach space and $G$ be a countable bounded subgroup of $GL(X)$ containing $-\Id$ and admitting a distinguished point. Does there exist an equivalent norm on $X$ for which $G$ is the group of isometries on $X$?
\end{quest}

That is,  for $X$ real separable, we ask whether any representation of a countable group $G$ by a bounded subgroup of $GL(X)$ containing $-\Id$ and admitting a distinguished point is always a display. Or in other words, we ask  whether it is possible to remove the Radon-Nikodym Property condition from Theorem \ref{FG}. This would follow if we could remove the RNP condition from Lancien's result, that is if every separable space could be renormed with an LUR norm without diminishing the group of isometries. This, however is false, as proved by the following examples:

\begin{ex} 
The space $C([0,1]^2)$ cannot be renormed with an equivalent LUR norm without diminishing the group of isometries.
\end{ex}

\begin{proof} 
It is well-known that if a norm on a Banach space $X$ is LUR, then the strong and the weak operator topology must coincide on ${\rm Isom}(X)$.
Indeed let $(T_n)$ be a sequence of isometries, not \textsc{sot}-converging to some isometry $T$, and fix $x \in B_X$ such that $T_nx$ does not converge in norm to $Tx$. We may assume that $\norm{Tx-T_nx} \geqslant \epsilon$ for some fixed $\epsilon$, so by the LUR property in $Tx$, $\norm{Tx+T_nx} \leqslant 2-\alpha$ for some fixed $\alpha>0$. Let then $x^*$ be some norm one functional such that $x^*(Tx)=1$.  Then
$$x^*(T_nx)=x^*(T_nx+Tx-Tx)\leqslant 1-\alpha=(1-\alpha) x^*(Tx),$$
and therefore $T_nx$ cannot converge weakly to $Tx$, and $T_n$ does not converge to $T$ in the weak operator topology. 

Now it was proved by Megrelishvili \cite{M} that the weak and the strong operator topology do not coincide on $G={\rm Isom}(C([0,1]^2))$. If $C([0,1]^2)$ could be renormed with an equivalent LUR norm, such that the group of isometries $G'$ in the new norm contained $G$, then by the above the weak and strong operator topology would coincide on $G'$ and therefore on the subgroup $G$, a contradiction.
\end{proof}

Not even a weak version of a theorem about LUR renormings not diminishing the group of isometries can be hoped for:

\begin{ex} Any equivalent renorming ${\norm{\cdot}}$ of the space $L_1([0,1])$, which does not diminish the group of isometries, is nowhere LUR. 
\end{ex}

\begin{proof} By \cite{FJ} Theorem 12.4.3, the norm on $L_1([0,1])$ is almost transitive, and therefore uniquely maximal. So any renorming which does not diminish the group of isometries must be a multiple of the original norm. On the other hand, Lemma \ref{atn} tells us that the norm of $L_1([0,1])$ is nowhere LUR.
\end{proof}

\section{Open questions and comments}

\subsection{Displays and distinguished points}

In general, the group of isometries on a separable Banach space need not be associated to a distinguished point or even a distinguished sequence. However, this question apparently remains open for countable groups.

\begin{quest} 
Let $X$ be a Banach space. Let $G$ be a countable group of isomorphisms on $X$ which is the group of isometries on $X$ in some equivalent norm. Must $X$ contain a point distinguished by $G$? Must $X$ contain a sequence distinguished by $G$? 
\end{quest}
 
In the separable case, a partial answer is that the countability of $G$ implies that there must be some $n$-uplet $x_1,\ldots,x_n$ of points and some $\alpha>0$ such that whenever $g \neq \Id$, $d(gx_i,x_i) \geqslant \alpha$ for some $i=1,\ldots,n$.

What we already observed  is that
the condition that there exist a distinguished point for a group $G$ implies that $G$ is discrete for SOT. When $X$ is separable, this also implies that $G$ is countable. Observe that in the metric space context, there does not need to exist a distinguished point, by an example  due to G. Godefroy:

\begin{ex}[G. Godefroy] There exists a complete separable metric space $M$ on which the group of isometries is countable and no point is distinguished by ${\rm Isom}(M)$.
\end{ex}

\begin{proof} Consider the union $M$ in $\R^2$ of the real line $\R \times \{0\}$ and of the segments  $\{q\}\times[0,1]$, for $q \in \Q$, with the metric defined by
$$d((x,y),(x',y'))=|y|+|x-x'|+|y'|.$$
It is clear that any isometry on $M$ is uniquely associated to an isometry on $\R$ which is either  a rational translation  or a symmetry with respect to a rational point, therefore ${\rm Isom}(M)$ is countable. 

Any point $m$ of $M$ is either a point of $\R \times \{0\}$, and then $\inf_{T \in {\rm Isom}(M) \setminus \{\Id\}}d(T(m),m)=0$, or a point of some segment $\{q\}\times[0,1]$, in which case the isometry associated to the symmetry with respect to $q$ will leave $m$ invariant. Therefore there is no distinguished point for ${\rm Isom}(M)$.  
\end{proof}

Observe however that there exists a sequence of two points which will be distinguished by ${\rm Isom}(M)$: any non-trivial "translation" isometry will send the point $(0,1)$ to a point at distance at least $2$, as well as any "symmetry" isometry with respect to a rational point which is not $0$. Finally to distinguish the symmetry isometry with respect to $0$, which fixes $(0,1)$, we may use the point $(1,0)$, which will also be sent to a point at distance $2$.

\subsection{Displays and universality}

Finally, we note that it remains unknown whether Theorem \ref{universal} admits one of the following generalisations.

\begin{quest} 
Let $G$ be a Polish group with a non-trivial central involution.  Is $G$ displayable on some  separable Banach space?
\end{quest}

\begin{quest} 
Does there exist a separable Banach space $X$ such that $\{-1,1\} \times G$ is displayable on $X$ for all Polish groups $G$?
\end{quest}

%\begin{quest} Does there exist a universal space for reflexive representability, that is, a Banach %space $X$ such that any reflexively representable topological (Polish) group is topologically %isomorphic to some subgroup of $Iso(X)$?
%\end{quest}


\begin{thebibliography}{99}
\bibitem{BR} J. Becerra Guerreo and A. Rodr\'iguez-Palacios, {\em Transitivity of the norm on Banach spaces}, Extracta Mathematicae Vol. 17 (2002), no. 1, 1--58.
\bibitem{Bel}  S. Bellenot, {\em Banach spaces with trivial isometries}, Israel J. Math. {\bf 56} (1986), no. 1, 89--96.
\bibitem{CS} F. Cabello-Sanchez, {\em Regards sur le probl\`eme des rotations de Mazur}, Extracta Math. {\bf 12} (1997), 97--116.
\bibitem{C} E. Cowie, {\em A note on uniquely maximal Banach spaces}, Proc. Amer. Math. Soc. {\bf 26} (1983), 85--87.
%\bibitem{day} M. Day, {\em Means for the bounded functions and ergodicity of the bounded representations of semi-groups}, Trans. Amer. Math. Soc. {\bf 69} (1950), 276--291.
\bibitem{DGZ} R. Deville, G. Godefroy and V. Zizler, {\em Smoothness and renormings in Banach spaces}, 
Pitman Monographs and Surveys in Pure and Applied Mathematics, 64. Longman Scientific and Technical, Harlow; copublished in the United States with John Wiley and Sons, Inc., New York, 1993.
%\bibitem{dix} J. Dixmier, {\em Les moyennes invariantes dans les semi-groupes et leurs applications}, Acta Sci. Math.
%Szeged {\bf 12} (1950), 213--227.
\bibitem{F} V. Ferenczi, {\em  Uniqueness of complex structure and real
    hereditarily indecomposable Banach spaces}, Advances in Math. {\bf 213},
    (2007),  no. 1, 462--488. 
\bibitem{FG} V. Ferenczi and E.M. Galego, {\em Countable groups of isometries on Banach spaces}, Transactions of the Am. Math. Soc. {\bf 362}(2010), no. 8, 4385--4431.
\bibitem{FR} V. Ferenczi and C. Rosendal, {\em On isometry groups and maximal symmetry}, preprint.
\bibitem{FJ} R. Fleming and J. Jamison, {\em Isometries on Banach spaces. Vol. 2. Vector-valued function spaces}, Chapman and Hall/CRC Monographs and Surveys in Pure and Applied Mathematics, 138. Chapman and Hall/CRC, Boca Raton, FL, 2008.
\bibitem{GK} S.  Gao and A.S. Kechris, {\em On the classification of Polish metric spaces up to isometry}, Mem. Amer. Math. Soc. 161 (2003), no. 766, viii+78 pp.

\bibitem{GM} W.T. Gowers and B. Maurey, {\em The unconditional basic sequence problem},  J. Amer. Math. Soc. {\bf 6} (1993), no. 4, 851--874.
\bibitem{J} K. Jarosz, {\em Any Banach space has an equivalent norm with trivial isometries}, Israel J. Math. {\bf 64} (1988), no. 1, 49--56.
\bibitem{kec} A. Kechris, {\em Classical descriptive set theory},  Graduate Texts in Mathematics {\bf 156}, Springer-Verlag, New York, 1995.
\bibitem{L} G. Lancien, {\em Dentability indices and locally uniformly convex renormings}, Rocky Mountain Journal of Mathematics {\bf 23}(1993), no. 2, 635--647.
\bibitem{LT}  J. Lindenstrauss and L. Tzafriri, {\em Classical Banach spaces I and II},  Lecture Notes in Mathematics, Vol. 338. Springer-Verlag, Berlin-New York, 1973.
\bibitem{M} M. Megrelishvili, {\em Operator topologies and reflexive representability}, Nuclear groups and Lie groups (Madrid, 1999), 197--208, Res. Exp. Math., 24, Heldermann, Lemgo, 2001.
\bibitem{R}  S. Rolewicz, {\em Metric linear spaces}, second. ed., Polish Scientific Publishers, Warszawa, 1984.
\bibitem{rossol}C. Rosendal and S. Solecki, {\em Automatic continuity of homomorphisms and fixed points on metric compacta}, Israel J. Math. {\bf 162} (2007), 349--371.
\bibitem{st}  J. Stern, {\em Le groupe des isom\'etries d'un espace de Banach} (French), Studia Math. {\bf 64}  (1979), no. 2, 139--149.
\bibitem{uspenskii} V.V. Uspenskij, {\em On the group of isometries of the Urysohn universal metric space}, 
Comment. Math. Univ. Carolin. 31 (1990), no. 1, 181--182.
\bibitem{weaver} N. Weaver, {\em Lipschitz algebras}, World Scientific Publishing Co., Inc., River Edge, NJ, 1999. xiv+223 pp. 
\end{thebibliography}
\end{document}